\documentclass[11pt,twoside]{amsart}

\usepackage[latin1]  {inputenc}%
\usepackage[T1]      {fontenc }%
\usepackage          {amsmath }%
\usepackage          {amsfonts}%
\usepackage          {amssymb }%
\usepackage          {amsthm  }%
\usepackage          {a4wide  }%
\usepackage          {url     }%
\usepackage          {tikz    }%
\usepackage[bookmarks=false,pdfborder={0 0 0.05}]{hyperref}
\usepackage[all]{xy}

\usepackage{enumerate, amsmath, amsfonts, amssymb, amsthm,  wasysym, graphics, graphicx, xcolor, frcursive,comment,bbm}

\usepackage{etex}
\newcommand{\ts}{\textsuperscript}

\definecolor{darkblue}{rgb}{0.0,0,0.7} % darkblue color
\definecolor{darkred}{rgb}{0.7,0,0} % darkred color
\usepackage{hyperref}
\usepackage[all]{xy}
\usepackage[T1]{fontenc}

\def\defn#1{{\sf #1}}

\newcommand{\RR}{\mathbb R}
\newcommand{\ZZ}{\mathbb Z}

\newcommand{\spanz}{\mathrm{span}_{\mathbb{Z}}}
\newcommand{\spanr}{\mathrm{span}_{\mathbb{R}}}
\newcommand{\lt}{\ell_{T}}

\newenvironment{patrick}
{\color{red}\begin{quote}\baselineskip=10pt\noindent
P: $\blacktriangleright$\ \small\sf} 
{\hskip1em\hbox{}\nobreak\hfill$\blacktriangleleft$\par
\end{quote}}

\newenvironment{thomas}
{\color{orange}\begin{quote}\baselineskip=10pt\noindent
T: $\blacktriangleright$\ \small\sf} 
{\hskip1em\hbox{}\nobreak\hfill$\blacktriangleleft$\par
\end{quote}}

\DeclareMathOperator{\Red}{Red}

\DeclareMathOperator{\Det}{det}
\DeclareMathOperator{\Rank}{rk}
\DeclareMathOperator{\idop}{id}
\DeclareMathOperator{\stab}{Stab}
\DeclareMathOperator{\Vol}{vol}

\newtheorem{theorem}{Theorem}[section]
\newtheorem{corollary}[theorem]{Corollary}
\newtheorem{Proposition}[theorem]{Proposition}
\newtheorem{Lemma}[theorem]{Lemma}

\theoremstyle{definition}
\newtheorem{Definition}[theorem]{Definition}
\newtheorem{remark}[theorem]{Remark}
\newtheorem{example}[theorem]{Example}
\newtheorem{Notation}[theorem]{Notation}

\newtheorem{question}[theorem]{Question}

\title[On the Hurwitz action in finite Coxeter groups]{On the Hurwitz action in finite Coxeter groups}

\author[B.~Baumeister]{Barbara Baumeister}
\address{Barbara Baumeister, Universit\"at Bielefeld, Germany}
\email{b.baumeister@math.uni-bielefeld.de}
\thanks{}

\author[T.~Gobet]{Thomas Gobet}
\address{Thomas Gobet, Technische Universit\"at Kaiserslautern, Germany}
{\color{red}\email{gobet@mathematik.uni-kl.de}}
\thanks{}

\author[K.~Roberts]{Kieran Roberts}
\address{Kieran Roberts, Universit\"at Bielefeld, Germany}
\email{kroberts@math.uni-bielefeld.de}
\thanks{}

\author[P.~Wegener]{Patrick Wegener}
\address{Patrick Wegener, Universit\"at Bielefeld, Germany}
\email{pwegener@math.uni-bielefeld.de}
\date{\today}

\begin{document}

\begin{abstract}
We provide a necessary and sufficient condition on an element of a finite Coxeter group to ensure the transitivity of the Hurwitz action on its set of reduced decompositions into products of reflections. We show that this action is transitive
if and only if the element is a parabolic quasi-Coxeter element, that is, if and only if it has a reduced decomposition 
into a product of reflections that generate a parabolic subgroup.
\end{abstract}

\maketitle

%\begin{patrick}
%I made some changes in section 5 to be able to give a proof of Lemma \ref{RootLatticeDeterminesRootSystem}, although I'm still not convinced of the part of the proof where we show equality of connexion indices. Add reference there...
%\end{patrick}

%\begin{thomas}
%What is the problem with the part about the connexion indices?
%\end{thomas}

%\begin{patrick}
%We are always very precise. Namely the following equation should be explained in more detail???
%$$\det(C_{\Phi})=\Vol(L(\Phi))= \Vol(L(\Phi(L(\Phi))))= \det (C_{\Phi(L(\Phi))}),$$
%\end{patrick}

%\begin{thomas}
%I don't think it is necessary.
%\end{thomas}

\tableofcontents

\section{Introduction}\label{sec:intro}

%\begin{patrick}
%$G$ arbitrary group, $n \geq 2$, there is an action of the Artin-braid group $\mathcal{B}_n$ on $G^n$ defined by 
%$$\sigma_i\cdot (g_1,\dots, g_n):=(g_1,\dots, g_{i-1}, g_i g_{i+1}g_i^{-1},g_i,g_{i+2},\dots, g_n).$$
%This action is called \defn{Hurwitz action} since it was first studied by Hurwitz in 1891 (\cite{Hur91}) in case $G=S_n$. Two elements $g,h \in G^n$ are called \defn{Hurwitz equivalent} if there is a braid $\sigma \in \mathcal{B}_n$ such that $\sigma(g)=h$. It has been shown by Liberman and Teicher (see \cite{LT}) that the question whether two elements in $G^n$ are Hurwitz equivalent or not, is in general undecidable. Nevertheless there are results for (semi)dihedral groups, dicyclic groups and generalized quaternion groups given by Hou and Sia (see \cite {Hou08}, \cite{Sia09}). In the case of a real reflection group, the Hurwitz action was studied by Humphries and Michel (see \cite{Hum03}, \cite{Mi06}). Recently Yaguchi (\cite{Ya15}) studied the Hurwitz action on $n$-tuples of distinct standard generators in $\mathcal{B}_{n+1}$. One reason to study the Hurwitz action comes from algebraic geometry, namely the braid monodromy of a projective curve (e.g. see \cite{KT00}). In terms of the Hurwitz action in Coxeter groups, this relation is described by Brieskorn (\cite{Bri88}). In the theory of Coxeter and Artin groups the Hurwitz action has applications to Garside theory (e.g. the transitivity of the H.a. on decompositions of Coxeter elements yields that dual Artin groups and Artin groups are isomorphic, see \cite{Mac14}).  
%\end{patrick}

This paper is concerned with the so-called dual approach to Coxeter groups. A \defn{dual Coxeter system} is a Coxeter group together with a generating set consisting of all the reflections in the group, that is, the set of conjugates of all the elements of a simple system. 

There are several motivations for studying dual Coxeter systems. They were introduced by Bessis \cite{Be03} and independently by Brady and Watt \cite{Br01, BW02}. The dual Coxeter systems are crucial in the theory of dual braid monoids, which are alternative braid monoids embedding in the Artin-Tits group attached to a finite Coxeter group (that is, a spherical Artin-Tits group) and providing an alternative Garside structure of it.
The latter allows a new presentation of the group and thereby for instance to get better solutions to the word problem in the spherical Artin-Tits groups (see \cite{BDM}, \cite{Be03}). Each dual braid monoid depends on a choice of a Coxeter element and  its poset of simple elements ordered by left-divisibility is isomorphic to the generalized noncrossing partition lattice with respect to that Coxeter element (see \cite{Be03}). 

%\begin{thomas}
%MacCammond is not in the spherical case. 
%\end{thomas}
%\begin{patrick}
%Nevertheless we should mention that Digne and MacCammond worked on the same problem for the affine Coxeter groups and obtained a complete result. Digne also showed transitive Hurwitz action at least in type $\widetilde{C}_n$. (As far as I know all of Jon's paper on this topic were recently accepted for publication).
%\end{patrick}
%\begin{thomas}
%We can do it if you want. But we are already citing $37$ papers (!) and our paper is definitely not concerned with Artin groups and even less of non-spherical type. The idea here was just to give this as a motivation since it is good to explain why the transitivity of the Hurwitz operation is important. I am adding the infinite case below after the complex case.  
%\end{thomas}

Bessis~\cite{Be14} later generalized these notions to complex reflections groups and their braid groups where the dual approach is natural. Indeed, there is in general no canonical choice of a simple system for a complex reflection group. 
%For affine Coxeter groups, in general the dual approach fails to give dual Garside structures (because the generalized noncrossing partitions fail to be a lattice in that case, see \cite{MaC15}) except in types $\widetilde{A}_n$ and $\widetilde{C}_n$ for some choices of Coxeter element (see~\cite{Dig1}, \cite{Dig2}).

%\begin{barbara}
%Maybe the affine case is getting to far. Why to mention it?
%\end{barbara}

Having replaced the set of simple generators of a Coxeter group by the whole set of reflections
in the Coxeter group (and in the Artin-Tits group), one needs to find new sets of relations between these new generators that define the respective groups. The idea is to take the so-called \defn{dual braid relations} \cite{Be03}. Unlike the classical braid relations, a dual braid relation can involve three generators and has the form $ab=ca$ (or $ba=ac$), where $a,b$ and $c$ are reflections. 

%As in the classical setting, that is, when the set of generators is given by a simple system, one introduces a length function with respect to the new set of generators (already studied by Carter \cite{Carter}), called reflection length function.

In the classical case Matsumoto's Lemma \cite{Mat} allows one to pass from any reduced decomposition of an element to any other one by successive applications of braid relations. The same question can be asked for reduced decompositions with respect to the new set of generators, and can be studied using the so-called \defn{Hurwitz action} on reduced decompositions. 

Let us say a bit more on this action. Let $G$ be an arbitrary group, $n \geq 2$. There is an action of the braid group $\mathcal{B}_n$ on $n$ strands on $G^n$ where the standard generator $\sigma_i\in\mathcal{B}_n$ which exchanges the $i$\ts{th} and $(i+1)$\ts{th} strands acts as 
$$\sigma_i\cdot (g_1,\dots, g_n):=(g_1,\dots, g_{i-1}, g_i g_{i+1}g_i^{-1},g_i,g_{i+2},\dots, g_n).$$
Notice that the product of the entries stays unchanged. This action is called the \defn{Hurwitz action} since it was first studied by Hurwitz in 1891 (\cite{Hur91}) in case $G=\mathfrak{S}_n$. 

Two elements $g,h \in G^n$ are called \defn{Hurwitz equivalent} if there is a braid $\beta \in \mathcal{B}_n$ such that $\beta(g)=h$. It has been shown by Liberman and Teicher (see \cite{LT}) that the question whether two elements in $G^n$ are Hurwitz equivalent or not is undecidable in general. Nevertheless there are results in many cases (see for instance \cite {Hou08}, \cite{Sia09}).
%For example for (semi)dihedral groups, dicyclic groups and generalized quaternion groups the Hurwitz action was studied by Hou and Sia (see \cite {Hou08}, \cite{Sia09}). In the case of a real reflection group see Michel 
Certainly the Hurwitz action also plays a role in algebraic geometry, more precisely in the braid monodromy of a projective curve (e.g. see \cite{KT00, Bri88} or \cite{JMS14}). 
 %In terms of the Hurwitz action in Coxeter groups, this relation is described by Brieskorn (\cite{Bri88}). 

In the case of finite Coxeter groups, the Hurwitz action can be restricted to the set of minimal length decompositions of a given fixed element $w$ into products of reflections. Given a reduced decomposition $(t_1,\dots, t_k)$ of $w$ where $t_1,\dots, t_k$ are reflections (that is, $w=t_1\cdots t_k$ with $k$ minimal), the generator $\sigma_i\in\mathcal{B}_n$ then acts as
$$\sigma_i\cdot (t_1,\dots, t_k):=(t_1,\dots, t_{i-1}, t_i t_{i+1}t_i,t_i,t_{i+2},\dots, t_k).$$

The right-hand side is again a reduced decomposition of $w$. In fact, we see that the braid group generator $\sigma_i$ acts on the $i$\ts{th} and $(i+1)$\ts{th} entries by replacing $(t_i, t_{i+1})$ by $(t_i t_{i+1}t_i, t_i)$ which corresponds exactly to a dual braid relation. Hence determining whether one can pass from any reduced decomposition of an element to any other just by applying a sequence of dual braid relations is equivalent to determining whether the Hurwitz action on the set of reduced decompositions of the element is transitive. 

The transitivity of the Hurwitz action on the set of reduced decompositions has long been known to be true for a family of elements commonly called \defn{parabolic Coxeter elements} (note that there are several unequivalent definitions of these in the literature). For more on the topic we refer to \cite{BDSW14}, and the references therein, where a simple proof of the transitivity of the Hurwitz action was shown for (suitably defined) parabolic Coxeter elements in a (not necessarily finite) Coxeter group. See also \cite{IS10}. The Hurwitz action in Coxeter groups has also been studied outside the context of parabolic Coxeter elements (see \cite{Voi85}, \cite{Hum03}, \cite{Mi06} - be aware that there are mistakes in \cite{Hum03}). 

The aim of this paper is to provide a necessary and sufficient condition on an element of a finite Coxeter group to ensure the transitivity of the Hurwitz action on its set of reduced decompositions. We call an element of a Coxeter group a \defn{parabolic quasi-Coxeter element} if it admits a reduced decomposition which generates a parabolic subgroup. 

The main result of this paper is the following characterization.
%The main result is the following.

\begin{theorem}
\label{th:maintheorem1}
  Let $(W,T)$ be a finite, dual Coxeter system of rank~$n$ and 
  let~$w \in W$.
  The Hurwitz action on $\Red_T(w)$ is transitive if and only if $w$ is a parabolic quasi-Coxeter element for $(W,T)$.
  %In symbols, for each $(t_1,\ldots,t_n) \in T^n$ such that 
  %$w = t_1 \cdots t_n$, there is a braid group element $b \in B_n$ such that
  %$$b(t_1,\ldots,t_n) = (s_1,\ldots,s_n).$$
\end{theorem}

The following theorem is an immediate consequence. 

\begin{theorem}
\label{th:maintheorem2}
Let $(W,T)$ be a finite, dual Coxeter system of rank~$n$ and 
  let~$w \in W$. If $w$ is a quasi-Coxeter element for $(W,T)$, then for \textbf{each} $(t_1 , \ldots , t_n) \in \Red_T(w)$ one has 
  $W= \langle t_1 , \ldots , t_n \rangle$. 
\end{theorem}

The proof that being a parabolic quasi-Coxeter element is a necessary condition to ensure the transitivity of the Hurwitz action is uniform. The other direction is case-by-case. In the simply laced types we first prove Theorem \ref{th:maintheorem2}
(see Theorem~\ref{ThmA}) and use it to prove the second direction of Theorem \ref{th:maintheorem1}.

Another consequence of Theorem \ref{th:maintheorem1} and of work of Dyer \cite[Theorem~3.3]{Dy90} is the following

\begin{corollary}
\label{th:theorem3}
Let $(W,T)$ be a dual Coxeter system and let $w \in W$. Further let
$(t_1 , \ldots , t_m) \in \Red_T(w)$, set $W^\prime := \langle t_1, \ldots , t_m\rangle$
and $T^\prime:= W^\prime \cap T$. If $W^\prime$ is finite, then the Hurwitz action on $\Red_{T^\prime}(w)$ is transitive.
\end{corollary}

It is an open question whether this statement remains true if there is no extra assumption on $W^\prime$.

The structure of the paper is as follows.
We adopt the approach and terminology introduced in [BDSW14] (see Section~\ref{sec:Dual}). 
In particular, we use an unusual definition of parabolic subgroup, which we show in 
Section~\ref{sec:parab} (after recalling some well-known facts on root systems in Section~\ref{sec:root}) to be equivalent to the classical definition for finite Coxeter groups and (as a consequence of results in [FHM06]) for a large family of irreducible infinite Coxeter groups, the so-called irreducible $2$-spherical Coxeter groups. As a byproduct, we obtain some results on parabolic subgroups of finite Coxeter groups. In particular, we show the following:

\begin{Proposition}\label{prop:parabmemes}
Let $(W,S)$ be a finite Coxeter system and $S'\subseteq T$ such that $(W,S')$ is a simple system. Then the parabolic subgroups with respect to $S$ coincide with those with respect to $S'$.
\end{Proposition}
%generalized definition of simple system).

%We make use of the approach and terminology introduced in \cite{BDSW14}, where in particular the definition of parabolic subgroups is not the usual one. Hence we first begin by showing that the classical definition of parabolic subgroups and that of \cite{BDSW14} coincide in the case where the Coxeter group is finite (Section \ref{sec:parab}); we also point out the equivalence of definitions for a large family of infinite Coxeter groups, as a consequence of results of W.~ N.~Franzsen, R.~B.~Howlett and B.~M\"uhlherr \cite{FHM06}. As a by-product we get some new results on parabolic subgroups of a finite Coxeter group, in particular we show that they do not depend on the chosen simple system (for the generalized definition of simple system given in \cite{BDSW14}). 

Given a root subsystem $\Phi^\prime$ of a given 
root system $\Phi$, we discuss in Section~\ref{sec:lattice} the relationship between the corresponding Coxeter groups and
root lattices, especially in the simply laced types. These results are needed later in Section~\ref{sec:quasi}.
It is known for the Coxeter groups of type $A_n$ that all the elements are parabolic Coxeter elements in the sense of \cite{BDSW14}. For types $B_n$ and $I_2(m)$, the sets of parabolic Coxeter elements and parabolic quasi-Coxeter elements coincide as it is shown in Section~\ref{sec:quasi}. In particular, Theorem \ref{th:maintheorem1} is true for $A_n, B_n$ and $I_2(m)$ as a consequence of \cite{BDSW14}. For the other types it is in general false that parabolic quasi-Coxeter elements coincide with parabolic Coxeter elements.  In Section~\ref{sec:quasi} we also show:

\begin{theorem}\label{thm:quasi-CoxeterParabolic}
Let $w$ be a quasi-Coxeter element in a finite dual Coxeter system $(W,T)$ of rank $n$ and $(t_1 , \ldots , t_n) \in \text{Red}_T(w)$ such that 
 $W = \langle t_1, \ldots , t_n \rangle$. Then the reflection subgroup $W' := \langle t_1 , \ldots , t_{n-1} \rangle$ is parabolic.
\end{theorem}

The proof of this theorem is uniform for the simply laced types, but case-by-case for the other ones. As a corollary we obtain a new characterization of the maximal parabolic subgroups of a finite dual Coxeter system (see Corollary~\ref{cor:CharParSub}). Moreover it follows from 
Theorem~\ref{thm:quasi-CoxeterParabolic} that an element is a parabolic quasi-Coxeter element if and only if it is a prefix of a quasi-Coxeter element (see Corollary~\ref{lem:parabquasicoxdividesII}).

Theorem~\ref{thm:quasi-CoxeterParabolic} allows us to argue by induction to prove the main theorem. For type $D_n$ we need to show that every two maximal 
parabolic subgroups intersect non-trivially provided that $n \geq 6$ (see Section~\ref{sec:Dn}) which then allows us to conclude Theorem \ref{th:maintheorem1} by induction. For the types $E_6, E_7$ and $E_8$ we first check by computer that every reflection occurs in the Hurwitz orbit of every reduced decomposition of a quasi-Coxeter element, and then argue by induction. Theorem \ref{th:maintheorem1} is verified for types $F_4$, $H_3$ and $H_4$ by computer. All this is done in Section~\ref{sec:proof}.
%
%In type $A_n$, it is known that every element of the group is a parabolic Coxeter element in the sense of \cite{BDSW14}. In types $B_n$ and $I_2(m)$, we show that the set of parabolic quasi-Coxeter elements coincide with the set of parabolic Coxeter elements in the sense of \cite{BDSW14}. This is done in type $B_n$ using the realization of the Coxeter group as a hyperoctahedral group and is easily seen in type $I_2(m)$. Hence in types $A_n$, $B_n$ and $I_2(m)$, $m\geq 3$, Theorem \ref{th:maintheorem1} follows from \cite{BDSW14}. However, it is false in general that parabolic quasi-Coxeter elements coincide with Coxeter elements. For types $F_4$, $H_3$ and $H_4$, the above theorem is checked by computer. For the simply-laced types other than $A_n$, there also exist parabolic quasi-Coxeter elements which are not parabolic Coxeter elements. In this case, we first  show that every prefix of a quasi-Coxeter element generates a parabolic subgroup. For the simply-laced types, we argue using the associated root systems. After that we show that every reflection occurs in the Hurwitz orbit of a quasi-Coxeter element (which is done by computer for the types $E_6, E_7, E_8$). With these ingredients we can then prove Theorem \ref{th:maintheorem1} in the simply-laced case. 
\bigskip\\
\textbf{Acknowledgments} 
The work was done while the third author held a position at the CRC 701 within the project C13 ''The geometry and combinatorics of groups''. The other authors wish to thank the DFG for its support through CRC 701 as well. The fourth author also thanks Bielefeld University for financial support through a scholarship by the rectorship. Last but not least we thank Christian Stump for fruitful discussions.

\section{Dual Coxeter systems and Hurwitz action} \label{sec:Dual}

\subsection{Dual Coxeter systems}\label{sub:dual}
Let $(W,T)$ be a \defn{dual Coxeter system} of finite rank $n$ in the sense 
of~\cite{Be03}. This is to say that there is a subset $S \subseteq T$ with $|S| = n$ such that $(W,S)$ is a (not necessarily finite) Coxeter system, and $T = \big\{ wsw^{-1} \mid w \in W, s \in S \big\}$ is the set of \defn{reflections} for the Coxeter system $(W,S)$ (unlike Bessis, we specify no Coxeter element). We call~$(W,S)$ a \defn{simple system} for $(W,T)$ and $S$ a set of \defn{simple reflections}. If $S'\subseteq T$ is such that $(W,S')$ is a Coxeter system, then $\big\{ wsw^{-1} ~|~w \in W, s \in S' \big\}=T$ (see \cite[Lemma 3.7]{BMMN02}). Hence a set $S'\subseteq T$ is a simple system for $(W,T)$ if and only if $(W,S')$ is a Coxeter system. The \defn{rank} of $(W,T)$ is defined as $|S|$ for a simple system $S$. This is well-defined by \cite[Theorem 3.8]{BMMN02}. It is equal to the rank of the corresponding root system.

Simple systems for $(W,T)$ have been studied by several authors (see \cite{FHM06}). Clearly, if $S$ is a simple system for $(W,T)$, then so is $wSw^{-1}$ for 
any $w \in W$. Moreover, it is shown in~\cite{FHM06} that for an important class of infinite Coxeter groups including the irreducible affine Coxeter groups, all simple 
systems for $(W,T)$ are conjugate to one another in this sense.

\noindent The following result is well-known and follows from \cite{Dy90}
\begin{Proposition}
Let $(W,S)$ be a (not necessarily finite) Coxeter system of rank $n$. Then $W$ cannot be generated by less than $n$ reflections.
\end{Proposition}

\begin{proof}
Assume that $W=\left\langle t_1,\dots, t_k\right\rangle$, with $k\leq n$. Following \cite{Dy90}, write $$\chi(W)= \{t \in T~|~\{u \in T~|~ \ell(ut) < \ell(t)\} = \{t\} \}$$ for the set of canonical Coxeter generators of $W$ where $\ell$ is the length function in $(W,S)$. It follows from \cite[Corollary 3.11 (i)]{Dy90} that  $|\chi(W)|\leq k$. But since $W$ is of rank $n$ the set $\chi(W)$ contains $S$ hence we have $|\chi(W)|\geq n$, which concludes the proof.
\end{proof}

A \defn{reflection subgroup} $W'$ is a subgroup of $W$ generated by reflections.
It is well-known that $(W',W'\cap T)$ is again a dual Coxeter system (see \cite{Dy90}).
For $w \in W$, a reduced $T$-decomposition of~$w$ is a shortest length decomposition of~$w$ into reflections, and we denote by $\Red_T(w)$ the set 
of all such reduced $T$-decompositions. When the context is clear, we drop the $T$ and elements of $\Red_T(w)$ are referred to as \defn{reduced decompositions} of $w$. The length of any element in $\Red_T(w)$ is called the \defn{reflection length} (or \defn{absolute length}) of $w$ and we denote it by $\lt(w)$. The absolute length function $\lt:W\rightarrow\mathbb{Z}_{\geq 0}$ can be used to define the \defn{absolute order} $<_T$ on $W$ as follows. For $u,v\in W$:
$$u<_T v\text{ if and only if }\lt(u)+\lt(u^{-1}v)=\lt(v).$$

%Similarly for a given simple system $(W,S)$, a reduced $S$-factorization of~$w$ is a shortest length factorization of~$w$ into simple reflections.
%An element $c \in W$ is called a \defn{parabolic Coxeter element} for $(W,T)$ if there is a simple system~$S = \{ s_1,\ldots,s_m \}$ such that
%$c = s_1\cdots s_n$ for some $n \leq m$.

The reflection subgroup generated by $\{s_1,\ldots,s_m\}$ is called a \defn{parabolic subgroup} for $(W,T)$ if there is a simple system~$S = \{ s_1,\ldots,s_n \}$ for $(W,T)$
with $m \leq n$. Notice that this differs from the usual notion of a parabolic subgroup generated by a conjugate of a subset of a fixed simple system $S$ (see \cite[Section 1.10]{Hu90}). However we prove in Section \ref{sec:parab} the equivalence of the definitions for finite Coxeter groups.

%we will nevertheless show in Section \ref{sec:parab} that for finite Coxeter groups, both definitions are equivalent.
%The element~$c$ is moreover called a \defn{standard parabolic Coxeter 
%element} for the Coxeter system $(W,S)$.
\subsection{Coxeter and quasi-Coxeter elements}

We now define (parabolic) Coxeter elements and introduce (parabolic) quasi-Coxeter elements. The second item of the definition below is borrowed from \cite{BDSW14}.

\begin{Definition}\label{def:coxeter} Let $(W, T)$ be a dual Coxeter system and $S=\{s_1,\dots, s_n\}$ be a simple system for $(W,T)$.
\begin{enumerate}

\item[(a)] We say that $c \in W$ is a \defn{classical Coxeter element} if $c$ is conjugate to some $s_{\pi(1)} \cdots s_{\pi(n)}$ for $\pi$ a permutation of the symmetric group $\mathfrak{S}_n$. An element $w \in W$ is a \defn{classical parabolic Coxeter element} if $w<_T c$ for some classical Coxeter element $c$.
\item[(b)] An element $c\in W$ is a \defn{Coxeter element} if there exists a simple system $S'=\{s_1',\dots, s_n'\}$ for $(W,T)$ such that $c=s_1'\cdots s_n'$. An element $w\in W$ is a \defn{parabolic Coxeter element} if there exists a simple system $S'=\{s_1',\dots, s_n'\}$ for $(W,T)$ such that $w=s_1'\cdots s_m'$ for some $m\leq n$.
\item[(c)] An element $w \in W$ is a \defn{quasi-Coxeter element} for $(W,T)$ if there exists $(t_1,\dots ,t_n)\in\Red_T(w)$ such that $W= \langle t_1 , \ldots , t_n \rangle$. An element $w \in W$ is a \defn{parabolic quasi-Coxeter element} for $(W,T)$ 
if there is a simple system $S'= \{ s_1' , \ldots , s_n'\}$ and $(t_1 , \dots , t_m)\in\Red_T(w)$ such that $\langle t_1 , \ldots , t_m \rangle = \langle s_1' , \ldots , s_m' \rangle$ for some $m \leq n$. 
\end{enumerate}
\end{Definition}

\begin{remark}\label{rmk:coxeter}
Let us point out a few facts about these definitions.
\begin{enumerate}
\item[(a)] An element of the form $s_{\pi(1)}\cdots s_{\pi(n)}$ obtained as product of elements of a fixed simple system in some order as in Definition \ref{def:coxeter} (a) is usually called a \defn{standard Coxeter element}. In case the Dynkin diagram of the Coxeter system is a tree, every two standard Coxeter elements are conjugate to each other by a sequence of cyclic conjugations (see \cite[V, 6.1, Lemme 1]{Bou81}), hence the set of classical Coxeter elements forms in that case a single conjugacy class. In particular, this holds for all the finite Coxeter groups. 
\item[(b)] If the Coxeter group is finite, then an element $w\in W$ is a parabolic Coxeter element if and only if $w<_T c$ for some Coxeter element $c$ (see \cite[Corollary 3.6]{DG}).
\item[(c)] It is clear that a classical Coxeter element is a Coxeter element and hence it follows from Remark \ref{rmk:coxeter} (b) that a classical parabolic Coxeter element is a parabolic Coxeter element. The difference between classical Coxeter elements and Coxeter elements is somewhat subtle. For finite Weyl groups (see 
Section~\ref{sec:root} for the definition), the two definitions are equivalent, as a consequence of \cite[Theorem 1.8(ii) and Remark 1.10]{RRS14}. It seems however that no case-free proof of this fact is known. An example where these two definitions differ is the dihedral group $I_2(5)$ (see \cite[Remark 1.1]{BDSW14}).

\item[(d)] We will show in Corollary  \ref{lem:parabquasicoxdividesII}  the same statement as the one given in Remark \ref{rmk:coxeter} (b) but for parabolic quasi-Coxeter elements, namely that $w\in W$ is a parabolic quasi-Coxeter element if and only if there exists a quasi-Coxeter element $w'\in W$ such that $w <_T w'$.
\end{enumerate}
\end{remark}

\begin{example}
In type $I_2(5)$ with Coxeter generators $s,t$, the classical Coxeter elements are the conjugates of $st$. The element $stst$ is a Coxeter element which is not classical. 
\end{example}

\begin{example}
In type $D_4$ with simple system $\{s_0, s_1, s_2, s_3\}$ where $s_2$ does not commute with any other simple reflection, the element $$s_1 (s_2 s_1 s_2) (s_2 s_0 s_2) s_3$$
is a quasi-Coxeter element. It has a reduced expression generating the whole group since if writing $$(t_1,\dots, t_4)=(s_1, s_2 s_1 s_2, s_2 s_0 s_2, s_3)$$
we have that $t_1 t_2 t_1=s_2$ and $s_2 t_3 s_2=s_0$. Using the permutation model for a group of type $D_4$ (see Section~\ref{sec:Dn}), it can be shown that there is no reduced expression of this element yielding a simple system for the group. By computer we checked that the poset $\{w\in W~|~w<_T c\}$ has $54$ elements and is not a lattice. There is a single conjugacy class of quasi-Coxeter elements which are not Coxeter elements in that case. 
\end{example}

\subsection{Rigidity}
Finite Coxeter groups are \defn{reflection rigid}, that is, if $(W,S)$ and $(W,S')$ are simple systems for $(W,T)$, then both systems determine the same diagram (see \cite[Theorem 3.10]{BMMN02}). , we define $(W,T)$ to be \defn{irreducible} if $(W,S)$ is irreducible and its type to be the type of $(W,S)$ for some (equivalently each) simple system $S \subseteq T$. In most cases, the type is determined by the group itself. There are only two exceptions, namely $W_{B_{2k+1}} \cong W_{A_1} \times W_{D_{2k+1}}$ $(k \geq 1)$ and $W_{I_{2}(4k+1)} \cong W_{A_1} \times W_{I_2(2k+1)}$ $(k \geq 1)$, which follows from the classification of the finite irreducible Coxeter groups and \cite[Theorem 2.17, Lemma 2.18 and Theorem 3.3]{Nu06}. 

\medskip

We say that a reflection group is \defn{strongly reflection rigid} if whenever 
$(W,S)$ and $(W,S')$ are simple systems for $(W,T)$, then $S$ and $S^`$ are conjugate in $W$.
    Notice that  $(W,T)$ is strongly reflection rigid if and only if every Coxeter element is a classical Coxeter element. Hence in particular let us point out that Remark~\ref{rmk:coxeter} (c) implies
    
    \begin{remark} The finite Weyl groups are strongly reflection rigid.
    \end{remark}

%\begin{theorem}
%\label{ThmNuida}
% (\cite[Thm. 2.17, Lemma 2.18 and Thm. 3.3]{Nu06})
% The only nontrivial direct product decompositions of a finite, irreducible Coxeter group $W$ are 
% \begin{enumerate}
%  \item $W(B_{2k+1}) \simeq W(A_1) \times W(D_{2k+1})$ $(k \geq 1)$;
%  \item $W(I_{2}(4k+1)) \simeq W(A_1) \times W(I_2(2k+1))$ $(k \geq 1)$;
 % \item For $W=W(E_7)$ resp. $W=W(H_3)$ it is $W=Z(W) \times W^+$, where $W^+$ denotes the normal subgroup of $W$ consisting
%  of elements of even length. In both cases $W^+$ is not a Coxeter group.
% \end{enumerate}
%\end{theorem}

\subsection{Hurwitz action on reduced decompositions}
The \defn{braid group} on $n$ strands denoted $\mathcal{B}_n$ is the group with generators 
$\sigma_1,\ldots , \sigma_{n-1}$ subject to the relations
\begin{align*}
 \sigma_i \sigma_j & = \sigma_j \sigma_i \quad \text{for } |i-j| > 1,\\
\sigma_i \sigma_{i+1} \sigma_i & = \sigma_{i+1} \sigma_i \sigma_{i+1} \quad \text{for } i=1,\dots,n-2.
\end{align*}
It acts on the set $T^n$ of $n$-tuples of reflections as
\begin{align*}
\sigma_i\cdot (t_1 ,\ldots , t_n ) &= (t_1 ,\ldots , t_{i-1} , \hspace*{5pt} t_i t_{i+1} t_i,
\hspace*{5pt} \phantom{t_{i+1}}t_i\phantom{t_{i+1}}, \hspace*{5pt} t_{i+2} ,
\ldots , t_n), \\
\sigma_i^{-1}\cdot (t_1 ,\ldots , t_n ) &= (t_1 ,\ldots , t_{i-1} , \hspace*{5pt} \phantom
{t_i}t_{i+1}\phantom{t_i}, \hspace*{5pt} t_{i+1}t_it_{i+1}, \hspace*{5pt} t_{i+2} ,
\ldots , t_n).
\end{align*}
\noindent We call this action of $\mathcal{B}_n$ on $T^n$ the \defn{Hurwitz action} and an orbit of this action an \defn{Hurwitz orbit}.

%For example, if $n=2$, then the action of $\sigma_{1}$ is described by 
%\begin{equation*}\ldots\mapsto(srs,srsrs)\mapsto (s,srs)\mapsto (r,s)\mapsto
%(rsr,r)\mapsto(rsrsr,rsr)\mapsto\ldots\end{equation*}
%for any $r,s\in T$. Note that in this case, the $B_{2}$-orbit of $(r,s)$ is the set 
%of all pairs  $(t_{1},t_{2})$ of reflections of the subgroup $\langle r,s\rangle$, 
%such that $t_{1}t_{2}=rs$. 

%\medskip

%The following lemma is a direct consequence of the definition.
\begin{Lemma}[{\cite[Lemma 1.2]{BDSW14}}]\label{lem:HurwitzOnFactorizations} Let $W'$ be a reflection subgroup of~$W$ and let $T' = T \cap W'$ be the 
  set of reflections in~$W'$.
  For an element $w \in W'$ with reduced $T$-factorization $w=t_1\cdots t_n$, the braid group on~$n$ strands acts on $\Red_{T'}(w)$.
\end{Lemma}

\medskip

%\noindent The results of this paper can be summarized by the following two Theorems

%\begin{theorem}
%\label{th:maintheorem1}
%  Let $(W,T)$ be a finite, dual Coxeter system of %rank~$m$ and 
%  let~$w \in W$.
%  The Hurwitz action on $\Red_T(w)$ is transitive if and only if $w$ is a parabolic quasi-Coxeter element for $(W,T)$.
  %In symbols, for each $(t_1,\ldots,t_n) \in T^n$ such that 
  %$w = t_1 \cdots t_n$, there is a braid group element $b \in B_n$ such that
  %$$b(t_1,\ldots,t_n) = (s_1,\ldots,s_n).$$
%\end{theorem}

%The following Theorem is a direct consequence of the Theorem above. But for the Coxeter groups of type $D_n, E_6, E_7, E_8$ we give an independent proof and  
%use Theorem \ref{th:maintheorem2} to prove Theorem \ref{th:maintheorem1}.

%\begin{theorem}
%\label{th:maintheorem2}
%Let $(W,T)$ be a finite, dual Coxeter system of rank~$n$ and 
%  let~$w \in W$. If $w$ is a quasi-Coxeter element for $(W,T)$, then for \textbf{each} $(t_1 , \ldots , t_n) \in \Red_T(w)$ one has 
%  $W= \langle t_1 , \ldots , t_n \rangle$. 
%\end{theorem}

We now give the main result of \cite{BDSW14}

\begin{theorem}
\label{theoremCoxeter}
  Let $(W,T)$ be a dual Coxeter system of finite rank~$n$ and 
  let~$c = s_1\cdots s_m$ be a parabolic Coxeter element in~$W$.
  The Hurwitz action on $\Red_T(c)$ is transitive.
  In symbols, for each $(t_1,\ldots,t_m) \in T^m$ such that 
  $c = t_1 \cdots t_m$, there is a braid $\beta \in \mathcal{B}_m$ such that
  $$\beta\cdot(t_1,\ldots,t_m) = (s_1,\ldots,s_m).$$
\end{theorem}

Hence in case $W$ is finite, Theorem \ref{th:maintheorem1} generalizes Theorem \ref{theoremCoxeter}. 

\begin{remark}\label{rmk:voigt}
For the case where $(W,T)$ is simply laced and $w$ is a quasi-Coxeter element in $W$, Voigt observed in his thesis \cite{Voi85} that the Hurwitz action on $\Red_T(w)$ is transitive. His definition of a quasi-Coxeter element is slightly different. Let $\Phi$ be the root system associated to $W$ (see Sections \ref{sec:root} and \ref{sec:lattice} for definitions and notations on root systems and lattices), then Voigt defined $w=s_{\alpha_1} \cdots s_{\alpha_n} \in W$ to be quasi-Coxeter if $\spanz(\alpha_1, \ldots, \alpha_n)$ is equal to the root lattice of $\Phi$. The connection with our definition will be explained in Section \ref{sec:lattice}.
\end{remark}

%The structure of this paper is as follows. In Section \ref{sec:root} we recall some facts on the geometry of root systems and fix the notations; in Section \ref{sec:parab} we show that the two definitions of parabolic subgroups are equivalent for finite Coxeter groups and uniformly prove one direction of Theorem \ref{th:maintheorem1}. We summarize some of the results about the 
%conjugacy classes of finite, crystallographic Coxeter groups given by Carter (see \cite{Carter}) in Section \ref{sec:conjugacy}. 
%We will use these results to prove Theorems \ref{th:maintheorem1} and \ref{th:maintheorem2} by a case-by-case analysis in Sections \ref{sec:quasi} and \ref{sec:proof}.

\section{Root systems and geometric representation}\label{sec:root}

In this section we fix some notation, and for the convenience of the reader we recall some facts on root systems and the geometric representation of a Coxeter group as can be found for example in \cite{Hu90} or \cite{Bou81}.

Let $V$ be a finite-dimensional Euclidean vector space with 
positive definite symmetric bilinear form $(-\mid-)$. For $0\neq\alpha \in V$, let 
$s_{\alpha}: V\rightarrow V$ be the reflection in the hyperplane orthogonal to $\alpha$, that is, the map defined by 

$$v \mapsto v - \frac{2(v\mid\alpha)}{(\alpha\mid\alpha)} \alpha.$$

Then $s_{\alpha}$ is an involution and $s_\alpha\in \mathrm{O}(V)$, the orthogonal group of $V$ with respect to $(-\mid-)$.

\medskip
\begin{Definition}\label{def:root}
A finite subset $\Phi \subseteq V$ of nonzero vectors is called \defn{root system} in $V$ if 
\begin{enumerate}
\item $\spanr(\Phi) =V,$
\item $s_{\alpha} ( \Phi ) = \Phi$ for all $\alpha \in \Phi,$
\item $\Phi \cap \RR \alpha = \{ \pm \alpha \}$ for all $\alpha \in \Phi.$
\end{enumerate}
The root system is called \defn{crystallographic} if in addition 
\begin{enumerate}
\item[(4)] $\left\langle \beta,\alpha\right\rangle:=\frac{2(\beta \mid \alpha)}{(\alpha \mid \alpha)} \in \ZZ$ for all $\alpha, \beta \in \Phi$. 
\end{enumerate}
The \defn{rank} $\Rank(\Phi)$ of $\Phi$ is the dimension of $V$. The group  $W_{\Phi}=\left\langle s_{\alpha}~|~\alpha\in\Phi\right\rangle$ associated to the root system $\Phi$ is a Coxeter group. In the case the root system is crystallographic then $W_\Phi$ is a \defn{(finite) Weyl group}. A subset $\Phi' \subseteq \Phi$ is called a \defn{root subsystem} if $\Phi'$ is a root system in $\spanr(\Phi')$.
\end{Definition}

Conversely, to any finite Coxeter group one can associate a root system. For an infinite Coxeter system $(W,S)$, one can still associate a set of vectors (again called a root system) with slightly relaxed conditions (see for instance \cite[Section 5.3-5.7]{Hu90}). In the following, when dealing with root systems it will always be in the sense of the above definition unless otherwise specified since this paper primarily concerns finite Coxeter groups. When results generalize to arbitrary Coxeter groups, we will mention that we work with the generalized root systems.

Let $\Phi \neq \varnothing$ be a root system. Then $\Phi$ is \defn{reducible} if $\Phi=\Phi_1\overset{\cdot}{\cup}\Phi_2$ where $\Phi_1$, $\Phi_2$ are nonempty root systems such that $(\alpha \mid \beta)=0$ whenever $\alpha\in\Phi_1$, $\beta\in\Phi_2$. Otherwise $\Phi$ is \defn{irreducible}. For an irreducible crystallographic root system $\Phi$, the set 
$\{ (\alpha \mid \alpha) \mid \alpha \in \Phi \}$ has at most two elements (see \cite[Section 2.9]{Hu90}. If this set has only one element, up to rescaling we can assume it to be $2$ and call $\Phi$ \defn{simply laced}. It follows from the classification of irreducible root systems that simply laced root systems are crystallographic. Simply laced root systems have types $A_n$ $(n\ge 1),D_n$ $(n\ge 4)$ or $E_n$ $(n\in\{6,7,8\})$ in the classification. We will sometimes use the notation $W_{X_n}$ for the Coxeter group with corresponding root system of type $X_n$ for convenience. For more on the topic we refer the reader to \cite{Hu90}.

\section{Equivalent definitions of parabolic subgroups}\label{sec:parab} 

In this section, we show one direction of Theorem \ref{th:maintheorem1} with a case-free argument. To this end, we show Proposition \ref{prop:parabmemes}, that is, we show that for finite Coxeter groups, the definition of parabolic subgroups given in Section \ref{sec:Dual} coincides with the usual one. We also mention that they coincide for a large class of infinite Coxeter groups.

Let $(W,S)$ be a finite Coxeter system with root system $\Phi$ and $V:=\spanr(\Phi)$ 
%(see Section \ref{sec:root} for notations on root systems). 
We say that a subgroup generated by a conjugate of a subset of $S$ is \defn{parabolic in the classical sense}. It is well-known that these are exactly the subgroups of the form $$C_W(E):=\{w\in W~|~w(v)=v~\text{for all }v\in E\}$$ where $E\subseteq V$ is any set of vectors (see for instance \cite[Section 5-2]{Kane}). 

\begin{Definition}\label{def:parabolic_closure}
Given a subset $\mathcal{A}\subseteq W$, the \defn{parabolic closure} $P_{\mathcal{A}}$ of $\mathcal{A}$ is the intersection of all the parabolic subgroups in the classical sense containing $\mathcal{A}$. It is again a parabolic subgroup in the classical sense (see \cite[12.2-12.5]{Ti74} or \cite{Qi07}). 
\end{Definition}

We denote by $\mathrm{Fix}(\mathcal{A})$ the subspace of vectors in $V$ which are fixed by any element of $\mathcal{A}$. If $\mathcal{A}=\{w\}$, then we simply write $\mathrm{Fix}(w)$ for $\mathrm{Fix}(\mathcal{A})=\ker(w-1)$ and $P_w$ for $P_\mathcal{A}$. For convenience we also set $\mathrm{Mov}(w):=\mathrm{im}(w-1)$. Note that $V=\mathrm{Fix}(w)\oplus \mathrm{Mov}(w)$ (see \cite[Definition 2.4.6]{Arm09}). It follows from the above description that $P_{\mathcal{A}}=C_W(\mathrm{Fix}(\mathcal{A}))$.

In this section we give a case-free proof that for finite Coxeter groups, the parabolic subgroups as defined in Section \ref{sub:dual} coincide with the parabolic subgroups in the classical sense. As a consequence we are able to show one direction of Theorem \ref{th:maintheorem1}. We first recall the following result

\begin{Lemma}[{\cite[Lemma 1.2.1(i)]{Be03}, \cite[Theorem 2.4.7]{Arm09} after \cite{Carter}}]\label{lem:inclusionfixed}
Let $w\in W$, $t\in T$. Then $$\mathrm{Fix}(w)\subseteq \mathrm{Fix}(t)~\text{if and only if }t<_{T}w.$$
\end{Lemma}

Notice that Lemma~\ref{lem:inclusionfixed} implies Theorem~1.4 of \cite{BDSW14}
if $W$ is finite.

\begin{Proposition}\label{PropEquiv}
Let $(W,S)$ be a finite Coxeter system, $T=\bigcup_{w\in W} wSw^{-1}$ and $w\in W$. If the Hurwitz action on $\Red_T(w)$ is transitive, then the subgroup generated by the reflections in any reduced decomposition of $w$ is equal to $P_w$.
\end{Proposition}
\begin{proof}
Let $(t_1,\dots, t_m)\in\Red_T(w)$ and assume that $W':=\left\langle t_1,\dots, t_m\right\rangle$ is not equal to $P_w$. Since $t_i<_{T} w$ for each $i$, we have $t_i\in P_w$ for all $i=1,\dots,m$ by Lemma \ref{lem:inclusionfixed}. It follows that $W'\subsetneq P_w$. Since both $W'$ and $P_w$ are reflection subgroups of $W$, there exists $t\in P_{w}$ with $t\notin W'$. It follows that $\mathrm{Fix}(w)\subseteq \mathrm{Fix}(t)$, hence also that $t<_T w$ by Lemma \ref{lem:inclusionfixed}. In particular there exists $(q_1,\dots,q_m)\in\Red_T(w)$ with $q_1=t$. Since the Hurwitz orbit of $(t_1,\dots, t_m)$ remains in $W'$ and $t\notin W'$, the Hurwitz action on $\Red_T(w)$ can therefore not be transitive.
\end{proof}

\begin{corollary}\label{cor:same}
Let $(W,S)$ be a finite Coxeter system and $T=\bigcup_{w\in W} wSw^{-1}$. A subgroup $P\subseteq W$ is parabolic if and only if it is parabolic in the classical sense. In particular, if $S'\subseteq T$ is such that $(W,S')$ is a simple system, then the parabolic subgroups in the classical sense defined by $S$ coincide with those defined by $S'$.
\end{corollary}
\begin{proof}
If $P$ is parabolic, then $P=\left\langle s_1,\dots, s_m\right\rangle$ where $\{s_1,\dots, s_n\}=S'\subseteq T$ is a simple system for $W$ and $m\leq n$. By \cite[Theorem 1.3]{BDSW14}, the Hurwitz action on $\Red_T(w)$ where $w=s_1s_2\cdots s_m$ is transitive. By Proposition \ref{PropEquiv}, it follows that $P$ is parabolic in the classical sense.

Conversely, if $P$ is parabolic in the classical sense, then $P$ is generated by a conjugate of a subset of $S$, and a conjugate of $S$ is again a simple system for $W$. Hence $P$ is parabolic.
\end{proof}

\noindent As a corollary we get a proof of one direction of Theorem \ref{th:maintheorem1}:

\begin{corollary}
\label{cor:onedirection}
  Let $(W,T)$ be a finite, dual Coxeter system of rank~$n$ and 
  let~$w \in W$.
  If the Hurwitz action on $\Red_T(w)$ is transitive, then $w$ is a parabolic quasi-Coxeter element for $(W,T)$.

\end{corollary}
\begin{proof}
Let $w=t_1\cdots t_m\in\Red_T(w)$. By Proposition \ref{PropEquiv}, $W':=\left\langle t_1,\dots, t_m\right\rangle$ is parabolic in the classical sense. By Corollary \ref{cor:same}, it follows that $W'$ is parabolic  and hence $w$ is a parabolic quasi-Coxeter element.
\end{proof}

If $(W,S)$ is of type $\tilde{A}_1$, i.e. if $W$ is a dihedral group of infinite order, then the non-trivial parabolic subgroups are precisely those subgroups that are generated by a reflection. Therefore parabolic subgroups in the classical sense coincide with parabolic subgroups in this case.
As an immediate consequence of a theorem of Franzsen, Howlett and M\"uhlherr \cite{FHM06} we also get the equivalence of the definitions for a large family of infinite Coxeter groups (including the irreducible affine Coxeter groups):

\begin{Proposition}
Let $(W,S)$ be an infinite irreducible $2$-spherical Coxeter system, that is $S$ is finite, and $ss'$ has finite order for every $s,s'\in S$. Then a subgroup of $W$ is parabolic if and only if it is parabolic in the classical sense. 
\end{Proposition}
\begin{proof}
Under these assumptions, if $(W,S')$ is a Coxeter system (without assuming $S'\subset T$), it follows from \cite[Theorem 1. b)]{FHM06} that there exists $w\in W$ such that $S'=wSw^{-1}$, hence any parabolic subgroup is a parabolic subgroup in the classical sense. 
\end{proof}

\begin{question}
Do parabolic subgroups always coincide with parabolic subgroups in the classical sense?
\end{question}

    \section{Root lattices}\label{sec:lattice}

In this section, we study root lattices and their  sublattices. The results will be needed for the better understanding of quasi-Coxeter elements in the simply laced types. 

\begin{Definition}
Let $V$ be an Euclidean vector space with symmetric bilinear form $(- \mid -)$. A \defn{lattice} $L$ in $V$ is the integral span of a basis of $V$. The lattice $L$ is called \defn{integral} if $(\alpha \mid \beta) \in \mathbb{Z}$ for all $\alpha, \beta \in L$ and
\defn{even} if $(\alpha \mid \alpha ) =2$ for all basis elements $\alpha$. 
\end{Definition}

For a set of vectors $\Phi\subseteq V$ we set $L(\Phi):=\spanz(\Phi)$.

\begin{remark}
 If $\Phi$ is a root system, then $L(\Phi)$ is a lattice, called 
 \defn{root lattice}.  If $\Phi$ is a crystallographic root system, then $L(\Phi)$ is an integral lattice.
\end{remark}

\begin{Proposition}
 For an even lattice $L$ the set 
 \begin{align*}
  \Phi(L):= \{\alpha \in L \mid (\alpha \mid \alpha) =2 \} 
 \end{align*}
 is a simply laced, crystallographic root system in $\spanr(L)$.
\end{Proposition}

\begin{proof}
 The set $\Phi(L)$ is contained in the ball around 0 with radius 2, therefore bounded, thus finite. The rest of the proof is straightforward.
\end{proof}

\begin{Definition}
Let $\Phi$ be a crystallographic root system in $V$. The \defn{weight lattice} $P(\Phi)$ of $\Phi$ is defined by
$$P(\Phi):=\{ x \in V \mid \left\langle x,\alpha\right\rangle \in \ZZ, ~\forall \alpha \in \Phi \}.$$
By \cite[VI, 1.9]{Bou81} it is again a lattice containing
$L(\Phi)$ and the group $P(\Phi) / L(\Phi)$ is finite. We call its order  the \defn{connection index} of $\Phi$ and denote it by $i(\Phi)$.
\end{Definition}

Note that if $\Phi$ is simply laced, the weight lattice is equal to the dual root lattice, namely,
\begin{align*}
 P(\Phi) = L^*(\Phi) := \{ x \in V \mid (x\mid y) \in \ZZ, ~\forall y \in L(\Phi) \}.
\end{align*}
\begin{Proposition}\label{ConnectionCartan}
Let $\Phi$ be a simply laced root system and let $C$ be the Cartan matrix of $\Phi$. Then
\begin{align*}
i(\Phi)=\Det(C).
\end{align*}
\end{Proposition}
\begin{proof}
Let $\Delta=\{\alpha_1 , \ldots , \alpha_m\}\subseteq\Phi$ be a basis of the root system $\Phi$. Then $\Delta$ is a basis of $L(\Phi)$. Denote by $M$ the Gram matrix of $L(\Phi)$ with respect to $\Delta$. By general lattice theory (see
for instance \cite[Section 1.1]{Eb02}) one has
 \begin{align*}
  \vert L^*(\Phi) : L(\Phi) \vert = \Det(M).
 \end{align*}
Since $\Phi$ is simply laced, we have $L^*(\Phi) = P(\Phi)$ and hence
 $i(\Phi) = \vert L^*(\Phi) : L(\Phi) \vert$. Again since $\Phi$ is simply laced, we have $C=M$, which concludes the proof.
\end{proof}

We list $i(\Phi)$ for the irreducible, simply laced root systems. These can be found in \cite[Planches I, IV,V,VI,VII]{Bou81}.
\begin{table}[h!]
\begin{tabular}{|l|l|l|l|l|l|}
\hline
 Type of $\Phi$ & $A_n$  & $D_n$  & $E_6$  & $E_7$  & $E_8$ \\ \hline
 $i(\Phi)$ & $n+1$ & 4 & 3 & 2 & 1  \\ \hline
\end{tabular}
\end{table}

As a consequence we obtain the following result.
\begin{Proposition}
Let $\Phi$ be an irreducible, simply laced root system. Then $\Phi$ is determined by the pair $(\Rank(\Phi), i(\Phi))$. 
\end{Proposition}

\noindent The following lemma seems to be folklore, but we could not find a proof in the literature, hence we state it here.
%For the following result we refer to \cite[Section 1.4]{He04}.
%For the following result we refer to Section 1.4 of G. Heckmann's lecture notes on Coxeter groups (\cite{He04}).

\begin{Lemma} \label{RootLatticeDeterminesRootSystem}
 Let $\Phi$ be a simply laced root system. Then the root lattice determines the 
 root system, that is, $$\Phi(L(\Phi))= \Phi.$$ 
\end{Lemma}

\begin{proof}
By the previous proposition, we have to show that the rank and connection indices of $\Phi$ and $\Phi(L(\Phi))$ coincide. Since $\Phi \subseteq \Phi(L(\Phi))$, we have $\Rank(\Phi)\leq \Rank(\Phi(L(\Phi)))$. On the other hand, the rank of $\Phi(L(\Phi)) $ is bounded above by the dimension of the ambient vector space which equals $\Rank(\Phi)$.

In the proof of Proposition \ref{ConnectionCartan}, we showed that the Cartan matrix of a root system and the Gram matrix of the corresponding root lattice coincide. Denote by $C_{\Phi}$ the Cartan matrix with respect to $\Phi$. Then
$$\det(C_{\Phi})=\Vol(L(\Phi))= \Vol(L(\Phi(L(\Phi))))= \det (C_{\Phi(L(\Phi))}),$$
which yields $i(\Phi)=i(\Phi(L(\Phi)))$ by Proposition \ref{ConnectionCartan}.
\end{proof}

\begin{Lemma} \label{RootsLength2}
Let $\Phi$ be as in \ref{RootLatticeDeterminesRootSystem} and $\Phi' \subseteq \Phi$ be a root subsystem. Then
  $L(\Phi') \cap \Phi = \Phi'$.
\end{Lemma}

\begin{proof} We have the equalities
 \begin{align*}
  L(\Phi') \cap \Phi \stackrel{\ref{RootLatticeDeterminesRootSystem}}{=} L(\Phi') \cap \{ \alpha \in L(\Phi) \mid (\alpha \mid \alpha) =2 \}
  = \{  \alpha \in L(\Phi') \mid (\alpha \mid \alpha)=2 \} \stackrel{\ref{RootLatticeDeterminesRootSystem}}{=} \Phi'.
 \end{align*}
\end{proof}

%\begin{Proposition} \label{subrootsystems}
% Let $\Phi$ be a finite root system, $R \subseteq \Phi$. Then $L(R) \cap \Phi$ and 
% $\spanr(R) \cap \Phi$ are subroot systems of $\Phi$ containing $R$.
%\end{Proposition}
%\medskip

\begin{Notation} Let $\Phi$ be a finite root system and $R\subseteq \Phi$ a set of roots. We denote by $W_R$ the group $\langle s_r~\vert~r \in R\rangle$.
\end{Notation}

Note that this is consistent with the notation introduced in Definition \ref{def:root} in case $R=\Phi$.

\begin{Proposition} \label{PropVoigt}
%(\cite[Prop. 1.6.1]{Voi85})
Let $\Phi$ be a crystallographic root system, 
$\Phi' \subseteq \Phi$ be a root subsystem, $R:= \{ \beta_1 , \ldots , \beta_k \} \subseteq  \Phi'$ be a set of roots. The following statements are equivalent:
 \begin{enumerate}
  \item[(a)] The root subsystem $\Phi'$ is the smallest root subsystem of $\Phi$ containing $R$ (i.e., the intersection of all root subsystems containing $R$).
  \item[(b)] $\Phi' =W_R(R)$.
  \item[(c)] $W_{\Phi'}= W_R$.
 \end{enumerate}
 Moreover, if any of the above equivalent conditions is satisfied, then 
 \begin{enumerate}
 \item[(d)] $L(\Phi') = L(R)$.
 \end{enumerate}
\end{Proposition}
%\begin{thomas}
%Where do we use the fact that the root system is simply-laced in the proof? I guess when we use that $L(\Phi')\cap\Phi=\Phi$ (false in type $B_2$). Maybe we could detail a bit. 
%\end{thomas}
%\begin{patrick}
%Yes, you're right. I added the argument. We should maybe add a remark to note that this proposition is in general false for non simply-laced
%\end{patrick}

%\begin{thomas}
%The notation $W(-)$ is used twice for different meanings; once it is the Weyl group of the root system inside and once the images of the roots under the Weyl group. I suggest to use the notation $W_R(R)$ for the action of the Weyl group on the set of roots and $W_\Phi$ for the Weyl group of the root system $\Phi$ (I changed in Prop. 3.7).  
%\end{thomas}

\begin{proof}
%  \item[(i) $\Rightarrow$ (iv)] Clearly $L(R) \subseteq L(\Phi')$. Assume $L(R) \neq L(\Phi')$. {\color{red} This implies that there exists a root $\alpha$ in $\Phi'$ which is not contained in $L(R)$.
%  By Proposition \ref{subrootsystems} one has that $L(R) \cap \Phi$ is a subroot system of $\Phi$ containing $R$, hence containing $\Phi'$. It follows that $\alpha\in L(R)$, a contradiction.}
%  \item[(iv) $\Rightarrow$ (i)] First of all
%  \begin{align*}
%   R \subseteq L(R) \cap \Phi = L(\Phi') \cap \Phi 
%   \stackrel{\ref{RootsLength2}}{=} \Phi'.
%  \end{align*}
%  Now assume that there exists a subroot system $\Phi''$ of $\Phi$ with 
%  $\Phi'' \subseteq \Phi'$ and $R \subseteq \Phi''$. {\color{red} It follows that} 
%  $L(\Phi') = L(R) \subseteq L(\Phi'')$. One then has
%  \begin{align*}
 %  \Phi' = L(\Phi') \cap \Phi \subseteq L(\Phi'') \cap \Phi 
 %  \stackrel{\ref{RootsLength2}}{=} \Phi'',
 % \end{align*}
 % Hence $\Phi'=\Phi''$.
  Obviously $W_R(R)$ is a root system with $R \subseteq W_R(R)$. Thus if (a) holds, then $\Phi^\prime \subseteq W_R(R)$.
  As $W_R(R)\subseteq W_{\Phi'}(\Phi')=\Phi'$, it follows that (a) implies (b). The converse direction follows from the definition of a root subsystem.

Statement (b) implies that $R\subseteq\Phi'$ and therefore we have $W_R \subseteq W_{\Phi'}$.
  We show that in this case $W_{\Phi'} \subseteq W_R$. To this end, let $\alpha \in \Phi' = W_R(R)$. Then 
  $\alpha= w(\beta_i)$ for some $w\in W_R$ and $1 \leq i \leq k$. Then 
  \begin{align*}
   s_{\alpha} = s_{w(\beta_i)}= w s_{\beta_i} w^{-1} \in W_R,
  \end{align*}
  which shows the claim. Thus (b) implies (c).
  
Assume (c) and let  $\Phi''$ be the smallest root subsystem of $\Phi$ containing $R$. Then $W_{\Phi'} = W_R\subseteq W_{\Phi''}$. By definition of $\Phi''$ we have $\Phi'' \subseteq \Phi'$. If $\Phi'' \subsetneq \Phi'$ then $W_{\Phi''}\subsetneq W_\Phi$, a contradiction. Hence $\Phi'' = \Phi'$, which shows $(a)$.

It remains to show that $(3)$ implies $(4)$. 
So assume $(c)$ and let $t_i:=s_{\beta_i}$, $1\leq i\leq k$. Let $T_{\Phi'}$ be the set of reflections in $W_{\Phi'}$. By \cite[Corollary 3.11 (ii)]{Dy90}, we have $T_{\Phi'}=\{wt_i w^{-1}~|~1\leq i\leq k, w\in W_{\Phi'}\}$. In particular any root in $\Phi'$ has the form $w(\beta_i)$ for some $w\in W_{\Phi'}$, $1\leq i\leq k$. Since $W_{\Phi'}=\left\langle t_1,\dots, t_k\right\rangle$, we can write $w=t_{i_1}\cdots t_{i_m}$ with $1\leq i_j\leq k$ for each $1\leq j\leq m$. Since $\Phi'$ is crystallographic it follows that $w(\beta_i)=t_{i_1}\cdots t_{i_m}(\beta_i)$ is an integral linear combination of the $\beta_j$'s, hence that $\Phi'\subseteq L(R)$. Since $R\subseteq \Phi'$ we get that $L(R)=L(\Phi')$, which proves (d).
\end{proof}

\begin{remark} \label{RemVoigt}
Notice that condition (d) in Proposition \ref{PropVoigt} is in general not equivalent  to conditions (a)-(c). For example, if $\Phi$ is of type $B_2$, then one can choose two orthogonal roots $\alpha$, $\beta$ generating a proper root subsystem of type $A_1\times A_1$ while one has $L(\{\alpha,\beta\})=L(\Phi)$. Nevertheless one has the equivalence for simply laced root systems:
\end{remark}
\begin{Lemma} \label{LemVoigt}
Let $\Phi, \Phi',R$ be as in Proposition \ref{PropVoigt} and assume in addition that $\Phi$ is simply laced. Then condition (d) in Proposition \ref{PropVoigt} is equivalent to any of the conditions (a), (b), (c).
\end{Lemma}
\begin{proof}
Since the conditions (a)-(c) are equivalent it suffices to show that (d) implies (a). Assume that there exists a root subsystem $\Phi''$ of $\Phi$ with 
$\Phi'' \subseteq \Phi'$ and $R \subseteq \Phi''$. It follows that 
$L(\Phi') = L(R) \subseteq L(\Phi'')$. One then has
\begin{align*}
\Phi' \stackrel{\ref{RootsLength2}}{=}  L(\Phi') \cap \Phi \subseteq L(\Phi'') \cap \Phi 
\stackrel{\ref{RootsLength2}}{=} \Phi'',
\end{align*}
Hence $\Phi'=\Phi''$, which concludes the proof.
\end{proof}

\medskip

\noindent The following two results are useful tools for the next section. The following proposition and its proof are borrowed from \cite[Prop. 1.5.2]{Voi85}.
\begin{Proposition} \label{FormelConnecIndex}
 Let $\Phi$ be a simply laced root system and $\Phi' \subseteq \Phi$ be a root subsystem with $\Rank(\Phi')=m=\Rank(\Phi)$. Then
 \begin{enumerate}
  \item[(a)] $\left| L(\Phi) : L (\Phi') \right| = \left| P(\Phi') : P(\Phi) \right|$
  \item[(b)] $\left| L(\Phi) : L (\Phi') \right| = \sqrt{i(\Phi') \cdot i(\Phi)^{-1}}$
 \end{enumerate}
\end{Proposition}

\begin{proof}   
 Since $\Phi$ and $\Phi'$ have the same rank, there is a lattice isomorphism 
 $\mathcal{L}: L(\Phi) \stackrel{\sim}{\longrightarrow} L(\Phi')$, 
 hence $| L(\Phi) : L (\Phi') | = \vert \text{det}(\mathcal{L}) \vert$
(see \cite[Section 1.1]{Eb02}). Furthermore we have
 \begin{align*}
  P(\Phi') & = \{ x \in V \mid (x\mid y) \in \mathbb{Z},~\forall y \in L(\Phi') \}
            = \{ x \in V \mid (x\mid\mathcal{L}(v)) \in \mathbb{Z},~\forall v \in L(\Phi) \}\\
           & = \{ x \in V \mid (\mathcal{L}^t(x)\mid v) \in \mathbb{Z},~\forall v \in L(\Phi) \}
            = \{ x \in V \mid \mathcal{L}^t(x) \in P(\Phi) \}\\
           & = \left( \mathcal{L}^t \right)^{-1}(P(\Phi)).
 \end{align*}
 Thus  $\left[P(\Phi') : P(\Phi) \right] = \vert \text{det}(\mathcal{L}^t)\vert = \vert \text{det}(\mathcal{L})\vert$, which shows (a).
 We have $$L(\Phi') \subseteq L(\Phi) \subseteq P(\Phi) \subseteq P(\Phi').$$ It follows that
 \begin{align*}
  | L(\Phi) : L (\Phi') | \underbrace{| P(\Phi) : L (\Phi) |}_{=i(\Phi)} 
  \underbrace{| P(\Phi') : P (\Phi) |}_{\stackrel{}{=}| L(\Phi) : L (\Phi') |}
  =\underbrace{| P(\Phi') : L (\Phi') |}_{=i(\Phi')},
 \end{align*}
which concludes the proof.
\end{proof}

\medskip

For a simply laced root system $\Phi$ we extend the definition of connection index to subsets of $\Phi$. For an arbitrary subset $R \subseteq \Phi$ we define $i(R):= | L^*(R): L(R) |$. Note that $i(R)$ is well-defined by 
Proposition \ref{PropVoigt}, because $i(R) = i(\Phi')$, where $\Phi'$ is the smallest root subsystem of $\Phi$
in $\spanr(R)$ containing $R$.

\medskip
The following theorem is part of the Diploma thesis of Kluitmann \cite{Kluitmann}. 

\begin{theorem} 
\label{EqualConnecIndex}
Let $\Phi$ be a simply laced root system. Let $w \in W_{\Phi}$ and let  $(s_{\alpha_1}, \ldots, s_{\alpha_k})$, $(s_{\beta_1}, \ldots , s_{\beta_k}) \in \Red_T(w)$, where $\alpha_i, \beta_i\in\Phi$, for $1\leq i\leq k$.  
 Then for $R:=\{ \alpha_1 , \ldots , \alpha_k\}$ and $Q:=\{ \beta_1 , \ldots , \beta_k \}$ we have
 \begin{align*}
  i(R) = i(Q).
 \end{align*}
\end{theorem}

%\begin{thomas}
%Notations in the Theorem, its proof and above should be uniformized; since we defined $L(-)$ in general there is no need here to use bold or calligraphic notation. 
%\end{thomas}

%\begin{patrick}
%Right! I changed letters to distinguish them from the frequently used letters $S$ and $T$.
%\end{patrick}
\begin{proof}
 Consider $L(R)$ and $L^*(R)$ (respectively $L(Q)$ and $L^*(Q)$) as lattices in $V':= \spanr(R)$ (respectively in $V'':= \spanr(Q)$). By \cite[Lemma 3]{Carter} the set $R$ is a basis for $L(R)$. Let 
 $\{\alpha^*_1 , \ldots , \alpha^*_k \}$ be the basis of $L^*(R)$ dual to $R$. In particular, $s_{\alpha_j}(\alpha^*_i) =\alpha^*_i$ for $i \neq j$ and 
 $s_{\alpha_i}(\alpha^*_i)= \alpha^*_i - \alpha_i$. Therefore
 \begin{align*}
  \theta_i:= (w - 1)(\alpha^*_i) = - s_{\alpha_1} \cdots s_{\alpha_{i-1}}(\alpha_i) \in L(R) \cap \Phi.
 \end{align*}

 Note that $\theta_1 = -\alpha_1$, $\theta_2 \in -\alpha_2+ \spanz(\alpha_1)$, and more generally $$\theta_i\in -\alpha_i+\spanz(\alpha_1,\dots, \alpha_{i-1}).$$ It follows that $\{\theta_1, \ldots , \theta_k\}$ is a basis of $V'$ and that 
 the map 
 \begin{align*}
  (w - 1)_{\vert V'}:L^*(R) \rightarrow L(R)
 \end{align*}
 is bijective.

Thus $i(R) = \vert \Det(w - 1)_{\vert V'} \vert$. The same argument with $Q$ instead of $R$ yields that
 $i(Q) = \vert \Det(w - 1)_{\vert V''} \vert$. By the proof of \cite[Theorem 2.4.7]{Arm09} we have that $R$ and $Q$ are both bases for $\text{Mov}(w)$, hence $V'=V'' = \text{Mov}(w)$ and hence 
 $i(R) = i(Q)$. 
\end{proof}

%\begin{thomas}
%I think that the theorem below as you had stated it before can be stated more generally; I keep your statement and proof of this remark and replace by a bunch of  generalized statements. Let me know what you think about all this. 

%\begin{theorem}
%\label{ThmA}
%Let $(W,T)$ be a finite dual Coxeter system of rank $m$ such that there exists $S \subseteq T$ with $(W,S)$ a simply-laced Coxeter 
%system. If $w$ is a quasi-Coxeter element for $(W,T)$, then for \textbf{each} 
%$(t_1, \ldots , t_m) \in \Red_T(w)$ \color{red}{we have} $ W= \langle t_1 , \ldots , t_m \rangle $.
%\end{theorem}

%\begin{proof}
% Let $\Phi$ be the root system corresponding to $(W,S)$ and $t_i := s_{\beta_i}$ for some $\beta_i \in \Phi$.
% Assume $W \neq \langle t_1 , \ldots , t_m \rangle$. By Prop. \ref{PropVoigt} this is equivalent 
% to assume that $L(\{ \beta_1 , \ldots , \beta_m  \}) \neq L(\Phi)$. Then $\Phi':= L(\{ \beta_1 , \ldots , \beta_m  \}) \cap \Phi$ is a proper subroot system 
% of $\Phi$. But by Proposition \ref{FormelConnecIndex} we have $i(\Phi') > i(\Phi)$, contradicting Theorem \ref{EqualConnecIndex}. 
%\end{proof}
%\end{thomas}

\section{Reflection subgroups related to prefixes of quasi-Coxeter elements}\label{sec:quasi}

In this section, we prove Theorem~\ref{th:maintheorem2} for simply laced dual
Coxeter systems, see Theorem~\ref{ThmA}.
Further we show that the reflections in a reduced $T$-decomposition of an element $w \in W$ generate a parabolic subgroup
whenever $w<_T w'$ for some quasi-Coxeter element $w'$.  Last not least we demonstrate that parabolic quasi-Coxeter elements coincide with parabolic Coxeter elements in types $A_n$, $B_n$ and $I_2(m)$.

Recall that for $w\in W$, we denote by $P_w$ the parabolic closure of $w$ (see Definition \ref{def:parabolic_closure}) and that $P_w =  C_{W}(\mathrm{Fix}(w))$.

\begin{theorem}
\label{ThmA}
Let $(W,T)$ be a simply laced dual Coxeter system of rank $n$. If $w\in W$ is a parabolic quasi-Coxeter element, then the reflections in any reduced $T$-decomposition of $w$ generate the parabolic subgroup $P_w$. That is, for each $(t_1,\ldots , t_m) \in \Red_T(w)$ we have $P_w= \langle t_1, \ldots, t_m \rangle$.
\end{theorem}
\begin{proof}
By the definition of a parabolic quasi-Coxeter element, there exists $(t_1,\dots, t_m)\in\Red_T(w)$ such that $P:=\langle t_1 , \ldots , t_m \rangle$ is a parabolic subgroup. By Lemma~\ref{lem:inclusionfixed}, we have $P \subseteq C_{W}(\mathrm{Fix}(w))=P_w$. Since $w\in P$, we have by definition of the parabolic closure that $P=P_w$. 
Let $(q_1,\dots, q_m)\in \Red_T(w)$. Then for all $1\leq i\leq m$ we have that $q_i <_T w$, which yields that $q_i$ is in $C_{W}(\mathrm{Fix}(w)) = P_w$. Thus $W':= \langle q_1, \ldots q_m\rangle$ is a subgroup of $P_w$.
 Let $\Phi$ be the root system of $P_w$ and $\beta_i\in\Phi$ be such that $q_i=s_{\beta_i}$, for $1\leq i\leq m$.  
 Then $L(\beta_1, \ldots,  \beta_m)$ is a sublattice of $L(\Phi)$ and $\Phi^\prime = L(\beta_1, \ldots , \beta_m) \cap \Phi$ is the smallest root subsystem of $\Phi$ that contains $\beta_1, \ldots, \beta_m$. Therefore
Theorem~\ref{EqualConnecIndex} yields that $ i(\Phi^\prime) = i(\Phi)$.
This implies $L(\Phi) = L(\Phi^\prime)$ by Proposition~\ref{FormelConnecIndex}.
Thus $W^\prime = W$ by Lemma~\ref{LemVoigt}.
% Assume that $W'$ is a proper subgroup of $P_w$. By Lemma~\ref{LemVoigt}, this is equivalent to assuming that $L(\{ \beta_1 , \ldots , \beta_m  \})$ is a proper sublattice of $L(\Phi)$. Then $\Phi':= L(\{ \beta_1 , \ldots , \beta_m  \}) \cap \Phi$ is a proper root subsystem 
% of $\Phi$ and by Proposition~\ref{FormelConnecIndex} we have $i(\Phi') > i(\Phi)$, contradicting 
% Theorem~\ref{EqualConnecIndex}.
\end{proof}

We will show in Corollary \ref{lem:parabquasicoxdividesII} that the following property of parabolic quasi-Coxeter elements does in fact characterize them.

\begin{Proposition}\label{lem:parabquasicoxdivides}
Let $(W,T)$ be a finite dual Coxeter system. If $w\in W$ is a parabolic quasi-Coxeter element, then there exists a quasi-Coxeter element $w'\in W$ such that $w<_T w'$.
\end{Proposition}
%is a parabolic quasi-Coxeter element, then there is a quasi-Coxeter element $w' \in W$ such that   $w<_{T} w'$.
%Conversely if $w'$ is a quasi-Coxeter element and if  $w <_{T} w'$, then $w$
%is a parabolic quasi-Coxeter element.
%\end{proposition}
\begin{proof}
Let $w\in W$ be a parabolic quasi-Coxeter element.
By definition, there exists a simple system $S=\{s_1,\dots, s_n\}$ for $W$ and a $T$-reduced decomposition $w=t_1\cdots t_m$ such that $\left\langle t_1,\dots, t_m\right\rangle =\left\langle s_1,\dots, s_m\right\rangle$, with $m\leq n$. Set $w':=t_1\cdots t_m s_{m+1}\cdots s_n$. It is clear that $$\left\langle t_1,\dots, t_m, s_{m+1},\dots, s_n\right\rangle=W.$$ Moreover we have $\lt(w')=n$, hence $w'$ is a quasi-Coxeter element with $w<_T w'$.
\end{proof}
        
%}
%\begin{proof}
% Let $\Phi$ be the root system corresponding to $(W,S)$ and $t_i := s_{\beta_i}$ for some $\beta_i \in \Phi$.
% Assume $W \neq \langle t_1 , \ldots , t_m \rangle$. By Prop. \ref{PropVoigt} this is equivalent 
% to assume that $L(\{ \beta_1 , \ldots , \beta_m  \}) \neq L(\Phi)$. Then $\Phi':= L(\{ \beta_1 , \ldots , \beta_m  \}) \cap \Phi$ is a proper subroot system 
% of $\Phi$. But by Proposition \ref{FormelConnecIndex} we have $i(\Phi') > i(\Phi)$, contradicting Theorem \ref{EqualConnecIndex}. 
%\end{proof}

%\begin{Corollary} \label{le:Bessis}
%(\cite[Lemma 1.4.3]{Be03}) Let $(W,T)$ be irreducible and crystallographic. An %element $w \in W$ is a parabolic Coxeter element in the classical sense if and only %if there exists some Coxeter element in the classical sense $c \in W$ with $w \leq_T %c$. 
%\end{Corollary}

%\begin{proof}
%Note that Bessis only works with bipartite Coxeter elements. But since Coxeter elements in the classical sense form a single conjugacy class and all simple systems for $W$ are conjugate, there exists $S \subseteq T$ such that $c$ is bipartite with respect to $S$. 
%\end{proof}

\begin{Lemma} \label{QuasiCoxIsCoxA}
Let $(W,T)$ be a dual Coxeter system of type $A_n$. Then each $w\in W$ is a classical parabolic Coxeter element.
\end{Lemma}

\begin{proof}
In type $A_n$, the set of $(n+1)$-cycles forms a single conjugacy class. Hence the set of classical Coxeter elements is exactly the set of $(n+1)$-cycles (see Remark \ref{rmk:coxeter} (a)). The assertion follows with
Remark \ref{rmk:coxeter} (b) as
for every element $w\in W$, we have $w<_{T} w'$ for at least one $(n+1)$-cycle $w'$. 
\end{proof}

\begin{Lemma}\label{lem:coxeterb}
Let $(W,T)$ be a dual Coxeter system of type $B_n$. Then every parabolic quasi-Coxeter element $w\in W$ for $(W,T)$ is a classical parabolic Coxeter element. 
\end{Lemma}

\begin{proof}
%Next we show that each minimal set of reflections generating $W_{B_n}$ gives rise to a Coxeter element (in the classical sense). 
For the proof we use the combinatorial description of $W_{B_n}$ as given in \cite[Section 8.1]{BB05}. Therefore let $S_{-n,n}$ be the group of permutations of $[-n,n]=\{-n,-n+1,\dots, -1,1,\dots, n\}$ and define
$$W=W_{B_n}:=\{w\in S_{-n,n}~|~w(-i)=-w(i),~\forall i\in[-n,n]\},$$
%the group of permutations $\pi$ of $[ \pm n ]$ such that $\pi(-i) = - \pi (i)$ for each $i \in [n]$,
also known as a \emph{hyperoctahedral group}. Then $(W,S)$ is a Coxeter system of type 
$B_n$ with 
\begin{align*}
S= \{ (1,-1), (1,2)(-1,-2), \ldots , (n-1, n)(-n+1, -n) \}.
\end{align*}
The set of reflections $T$ for this choice of $S$ is given by 
\begin{align*}
T= \{ (i,-i) \mid i \in [n]\}
\cup \{ (i,j)(-i,-j) \mid 1 \leq i < \vert j \vert \leq n\}.
\end{align*}

%\begin{Lemma}\label{lem:coxeterb}
%Let $(W,T)$ be a dual Coxeter system of type $B_n$. Then every parabolic quasi-Coxeter element $w\in W$ for $(W,T)$ is a parabolic Coxeter element. 
%\end{Lemma}
We show that every quasi-Coxeter element for $(W,T)$ is a classical Coxeter element. If $w$ is a parabolic quasi-Coxeter element, then 
by Proposition~\ref{lem:parabquasicoxdivides}, there exists a quasi-Coxeter element $w'\in W$ such that $w<_T w'$. Hence if $w'$ is a classical Coxeter element, then $w <_T w'$ implies that $w$ is a classical parabolic Coxeter element.

Let $R=\{r_1,\dots, r_n\} \subseteq T$ be such that $\langle R \rangle =W$. It suffices to show that $r_1 r_2\cdots r_n$ is in fact a classical Coxeter element.

%We will show that $r_1 r_2\cdots r_n$ is in fact a classical Coxeter element.

The group $W$ cannot be generated only by reflections of type $(i,j)(-i,-j)$, $i\neq \pm j$. Therefore there exists $i \in [n]$ with $(i,-i) \in R$. If there exists $j \in [n], j \neq i$ with
$(j,-j) \in R$, then $R$ cannot generate the whole group $W$. 
%We want to show that $r_1r_2\cdots r_n$ is a Coxeter element in the classical sense. 
Since classical Coxeter elements are closed under conjugation, we can conjugate the set $R$ with 
$(1,i)(-1,-i)$ (if necessary) and assume $(1,-1) \in R$. 

Since $R$ generates the whole group $W$, there does not exist $j \in [n]$ which is fixed by each $r \in R$. Thus for each $k \in [n], k \neq 1$, we can find $i_k \in [\pm n]$ with $k \neq \pm i_k$ such that $(k,i_k)(-k,-i_k) \in R$. Therefore 
\begin{align*}
R= \{ (1,-1), (2,i_2)(-2,-i_2), \ldots , (n,i_n)(-n, -i_n) \}.
\end{align*}
%If the set $\{\vert i_2 \vert , \ldots , \vert i_n \vert \}$ has cardinality $n-1$, it is straightforward to show that $R$ is a simple system for $W$.

Note that some $i_j$ has to equal $\pm 1$, because otherwise $(1,-1)$ would commute with any element of $W$. By conjugating $R$ with $(j,2)(-j,-2)$ resp. $(j,-2)(-j,2)$ (if necessary) we can assume that $i_2=1$. Hence after rearrangement we can assume that $R$ is of the form
\begin{align*}
R=\{ (1,-1), (2,1)(-2,-1), (3,i_3)(-3,-i_3), \ldots , (n,i_n)(-n, -i_n) \}.
\end{align*}
Similarly to what we did above, there exists $j \geq 3$ with $i_j \in \{\pm 1 , \pm 2\}$.
By conjugating $R$ with $(j,3)(-j,-3)$ resp. $(j,-3)(-j,3)$ (if necessary) we can assume that $i_3 \in \{ 1,2 \}$. Continuing in this manner we obtain 
\begin{align*}
R= \{ (1,-1), (2,i_2)(-2,-i_2), \ldots , (n,i_n)(-n, -i_n) \}
\end{align*}
with $i_j \in \{ 1, \ldots , j-1 \}$ for each $j \in \{2, \ldots, n\}$. A direct computation shows that $c:= r_1r_2\cdots r_n$ is a $2n$-cycle and thus a classical Coxeter element. Indeed, there is a single conjugacy class of $2n$-cycles in $W$. 
%Therefore both notions coincide in type $B_n$. 

\end{proof}

\begin{remark}
Notice that by Remark~\ref{rmk:coxeter}~(c), it is already known that classical Coxeter elements and Coxeter elements must coincide in type $B_n$. Moreover it follows from \cite[Lemma 8, Theorem A]{Carter} that every quasi-Coxeter element is actually a Coxeter element. Hence on can derive Lemma \ref{lem:coxeterb} from these two observations. However since both of them rely on sophisticated methods, we prefered to give here a direct proof using the combinatorics of the hyperoctahedral group.
\end{remark}

\begin{remark}
In type $A_n$, we even have that every element $w$ such that $\lt(w)=n$ is a classical Coxeter element (thus quasi-Coxeter), because such an element is necessarily an $(n+1)$-cycle. Notice that this fails in type $B_n$. For instance, the product $(1,-1)(2,-2)\cdots (n,-n)$ in $W_{B_n}$ has reflection length equal to $n$, but it is not a quasi-Coxeter element. 
\end{remark}

%For the rest of this section we work with the combinatorial realization of $W:=W_{D_n}$ as a subgroup of the group $S_{-n,n}$ of permutations of $[-n,n]=\{-n,-n+1,\dots, -1,1,\dots, n\}$ as given in \cite[Section 8.2]{BB05}. That is

%\begin{Lemma} \label{QuasiCoxIsCoxB}
%Let $(W,T)$ be a dual Coxeter system of type $B_n$. Then any parabolic quasi-Coxeter element $w\in W$ for $(W,T)$ is a parabolic Coxeter element.
%\end{Lemma}

%\begin{proof}
%By Proposition~\ref{lem:parabquasicoxdivides}, there exists a quasi-Coxeter element $w'\in W$ such that $w<_T w'$ and by \cite[Theorem A]{Carter}every quasi-Coxeter element of type $B_n$ is a Coxeter element. {\color{red}Hence $w'$ is a Coxeter element and since $w <_T w'$ it follows by Remark \ref{rmk:coxeter}~$(b)$ that $w$ is a parabolic Coxeter element.}
%\end{proof}

The following is well-known (see \cite[IV, 1.2, Proposition 2]{Bou81}):

\begin{Proposition}
A group $W$ is a dihedral group if and only if it is generated by two elements $s$, $t$ of order $2$, in which case $\{s,t\}$ is a simple system for $W$. 
\end{Proposition}

\begin{corollary}\label{QuasiCoxIsCoxI}
Let $(W,T)$ be a dual Coxeter system of type $I_2(m)$, $m\geq 3$. Then $w$ is a quasi-Coxeter element in $W$ if and only if $w$ is a Coxeter element in $W$. It follows that $w\in W$ is a parabolic quasi-Coxeter element if and only if $w$ is a parabolic Coxeter element. 
\end{corollary}

Note that Coxeter elements and classical Coxeter elements do not coincide in general in dihedral type (see Remark \ref{rmk:coxeter} (c)).
%In the next two sections we will sometimes use the notation $W_{X_n}$ for the Coxeter group of type $X_n$ for convenience.
\bigskip\\
\noindent
\begin{proof}[Proof of Theorem~\ref{thm:quasi-CoxeterParabolic}]
The reduction to the case where $W$ is irreducible is immediate. The proof is uniform for the simply laced types and case-by-case for the remaining exceptional groups.

\medskip

(Dihedral type) The claim is obvious in that case.

\medskip

(Simply laced types) Let $\Phi$ be a root system for $(W,T)$ with ambient vector space $V$.
 Let $P_{W'}$ be the parabolic closure of $W'$. For $1\leq i\leq n$, let  $\beta_i \in \Phi$ be a root corresponding to $t_i$ and let
 $\Phi'$ be the root subsystem of $\Phi$ generated by $R:= \{\beta_1, \ldots , \beta_{n-1}\}$ so that $W_R=W_{\Phi'}=W'$ (see Proposition \ref{PropVoigt}). Let $\Psi\subseteq \Phi$ be the root subsystem of $\Phi$ associated to $P_{W'}$. By \cite[Lemma 3]{Carter} the set $R\cup \{\beta_n\}$ is a basis of $V$.

 Let $U$ be the ambient vector space for $\Psi$. As the linear independent set $R$ is a subset of $\Psi$, the dimension of $U$ is at least $n-1$. Since $P_{W'}$ is the parabolic closure of $t_1\cdots t_{n-1}$ it has to be the centralizer of a line in $V$ and therefore $\dim U = n-1$. It follows that $U= \spanr(\beta_1, \ldots, \beta_{n-1})$.

By Proposition \ref{PropVoigt} we have that $L(\Phi')=L(\{ \beta_1, \ldots , \beta_{n-1}\})$ and $L(\Phi)=L(\{ \beta_1, \ldots , \beta_n\})$.
 %$$L(\phi') = \{\sum_{i = 1}^{n-1} \lambda_i \beta_i~\vert \lambda_i \in \ZZ, 1 \leq i \leq n-1\}~\mbox{and}~L(\phi) = \{\sum_{i = 1}^{n} \lambda_i \beta_i~\vert \lambda_i \in \ZZ, 1 \leq i \leq n\}.$$
Since $V = U \oplus \RR\beta_n$ we have
 \begin{align*}
  L(\Phi) \cap U = L(\{\beta_1, \ldots, \beta_{n-1}\}) =  L(\Phi').
 \end{align*}
 As $L(\Psi) \subseteq U$, it follows that $L(\Psi) \subseteq L(\Phi')$.
 But since $\Phi' \subseteq \Psi$ we get that 
 $L(\Phi') \subseteq L(\Psi)$ and therefore $L(\Phi') = L(\Psi)$. Thus 
 $W' = P_{W'}$ by Proposition \ref{PropVoigt}.
 
  \medskip

(Type $B_n$) By Lemma \ref{lem:coxeterb} the element $w$ is a classical Coxeter element. It follows that $wt_n$ is a classical parabolic Coxeter element, hence a parabolic Coxeter element (see Remark \ref{rmk:coxeter} (c)). It follows that $W'$ is parabolic.

 \medskip
 
% \begin{thomas}
% I did not understand below in type $F_4$ why it is enough to check that $w'$ cannot be in $W_{A_3}$. In \cite[Table 6]{DPR14} there are other reflection subgroups of rank $3$ which are not parabolic for instance of type $B_2\times A_1$, is there an obvious way of getting rid of them?
%\end{thomas}
 
(Type $F_4$) By \cite[Table 6]{DPR14}, the group $W_{F_4}$ cannot be generated by just adding one reflection to one of the non-parabolic rank $3$ reflection subgroups.  
%Let $w'=wt_4$. In \cite[Table 6]{DPR14}, the reflection subgroups of a Coxeter group of type $F_4$ and their parabolic closures are determined. We only have to check that $w'$ cannot be quasi-Coxeter in a rank $3$ reflection subgroup isomorphic to $W_{A_3}$. Using \cite{GAP2015} we found 384 sets containig three reflections which generate a reflection subgroup isomorphic to $W_{A_3}$. None of these can be completed to a generating set of $W_{F_4}$ by just adding one reflection.
%  \begin{barbara}
% Why not to quote Carter here? This can everybody check easily.
% \end{barbara}
 
% \begin{itemize}
%  \item[(i)] $W'$ is of type $A_3$. Therefore $w'$ is Coxeter in $W'$ by %$Proposition \ref{quasi-CoxeterAB}. Let $t_i$ be the reflection in the root
 % $\beta_i$. By Theorem \ref{theoremCoxeter} we can choose $\beta_1, \beta_2 , \beta_3$ to be a simple system. We take 
%  $\beta_i = e_i - e_{i+1}$ ($1 \leq i \leq 3$). Therefore $\beta_4$ has to be a short root. If $\beta_4 = \pm e_i$ ($1 \leq i \leq 4$) the root system 
%  $\langle \beta_1 , \ldots , \beta_4 \rangle $ is of type $B_4$. If $\beta_4 = \frac{1}{2} (\pm e_1 \pm e_2 \pm e_3 \pm e_4)$, then 
%  $\langle \beta_1 , \ldots , \beta_4 \rangle $ is of type $A_3 \times \widetilde{A}_1, \ldots$
%  \item[(ii)] $W'$ is of type $\widetilde{A}_1 \times B_2$ ...
% \end{itemize}

\medskip

(Types $H_3$ and $H_4$) We refer 
to \cite[Tables 8 and 9]{DPR14}, where the reflection subgroups of $W_{H_3}$ and $W_{H_4}$ and their parabolic closures are determined.
\begin{enumerate}
\item Each rank $2$ reflection subgroup of the group $W_{H_3}$ is already parabolic.
\item The only rank $3$ reflection subgroup of $W_{H_4}$ that is not parabolic has type $A_1 \times A_1 \times A_1$. Taking a set of three reflections generating such a reflection subgroup, we checked using \cite{GAP2015} that this set cannot be completed to obtain a generating set for $W_{H_4}$ by adding a single reflection.
\end{enumerate}

\end{proof}

\begin{remark}\label{counterexample}
Theorem~\ref{thm:quasi-CoxeterParabolic} is not true in general. It can even fail if $w$ is a Coxeter element, as the following example borrowed from \cite[Example 5.7]{HK13} shows: Let $W=\left\langle s,t,u\right\rangle$ be of affine type $\widetilde{A}_2$, and let $w=stu$. Then $s(tut)<_T c$, but $W'=\left\langle s, tut\right\rangle$ is an infinite dihedral group, hence it is not a parabolic subgroup in the classical sense since proper parabolic subgroups of $W$ are finite. We mentioned in Section \ref{sec:parab} that for affine Coxeter groups parabolic subgroups in the classical sense also coincide with parabolic subgroups as defined in Subsection \ref{sub:dual}. 
\end{remark}

\begin{corollary} \label{cor:CharParSub}
Let $(W,T)$ be a finite dual Coxeter system of rank $n$ and $W'$ a reflection subgroup of rank $n-1$. Then $W'$ is a parabolic subgroup if and only if it exists $t \in T$ such that $\langle W', t \rangle = W$.  
\end{corollary}

\begin{proof}
The necessary condition is clear by the definition of parabolic subgroup. The sufficient condition is a direct consequence of Theorem \ref{thm:quasi-CoxeterParabolic}.
\end{proof}

We also derive a characterization of parabolic quasi-Coxeter elements analogous to that of parabolic Coxeter elements (see Remark \ref{rmk:coxeter} (b)).

\begin{corollary}\label{lem:parabquasicoxdividesII}
Let $(W,T)$ be a finite dual Coxeter system and $w\in W$. Then $w$ is a parabolic quasi-Coxeter element if and only if there exists a quasi-Coxeter element $w'\in W$ such that $w<_T w'$.
\end{corollary}
%is a parabolic quasi-Coxeter element, then there is a quasi-Coxeter element $w' \in W$ such that   $w<_{T} w'$.
%Conversely if $w'$ is a quasi-Coxeter element and if  $w <_{T} w'$, then $w$
%is a parabolic quasi-Coxeter element.
%\end{proposition}
\begin{proof} 
The forward direction is given by Proposition~\ref{lem:parabquasicoxdivides}.
Now let $w<_T w'$, where $w'\in W$ is a quasi-Coxeter element. Using Theorem \ref{thm:quasi-CoxeterParabolic} inductively we get that $w$ is a parabolic quasi-Coxeter element in $W$.
\end{proof}

%\begin{proof}
%To see this in type $A_n$ consider an admissible factorization $w=t_1 \cdots t_k$. By Theorem \ref{ThmA} $\langle t_1 , \ldots , t_k \rangle =P$ for some parabolic subgroup $P$. In particular Theorem \ref{ThmA} implies that $w$ is not contained in any proper reflection Subgroup $W'$ of $P$. Furthermore in type $A_n$ there are no admissible diagrams containing a cycle. Hence the admissible diagram corresponding to the admissible factorization has to be the Dynkin diagram itself and $w$ is a parabolic Coxeter element in the classical sense. In type $B_n$ each parabolic subgroup has a Dynkin diagram whose connected components are of type $A$ or $B$. Therefore the assertion follows from the appendix.
%\end{proof}

\begin{remark}
Corollary~\ref{lem:parabquasicoxdividesII} does not hold for infinite Coxeter groups as it fails for the Coxeter element given in Remark~\ref{counterexample}.
\end{remark}

\section{Intersection of maximal parabolic subgroups in type $D_n$}\label{sec:Dn}

The aim of this section is to show the following result which will be needed in the next section in the proof of Theorem \ref{th:maintheorem1}.

\begin{Proposition} \label{Prop:AppB}
Let $(W,S)$ be a Coxeter system of type $D_n$ ($n \geq 6$). Then the intersection of two maximal parabolic subgroups is non-trivial.
\end{Proposition}

\begin{remark}
This statement is not true in general, not even in the simply laced case. In particular, it fails in types $D_4, D_5, E_7$ and $E_8$. For example
\begin{enumerate}
\item[(a)] Let $(W,S)$ be of type $D_4$ where $S=\{s_0, s_1,s_2,s_3\}$ with $s_2$ commuting with no other simple reflection, then both $P:=\left\langle s_0, s_1, s_3\right\rangle$ and $s_2 P s_2$ are maximal parabolic subgroups of type $A_1\times A_1\times A_1$ and have trivial intersection.

\item[(b)] Let $(W,S)$ be of type $E_7$ where $S=\{ s_1, \ldots , s_7 \}$ labelled as in \cite[Planche VI]{Bou81}.
Let $I= \{ s_1, s_2 , s_3, s_4 , s_6, s_7\}$ and $J= \{ s_1, s_2 , s_3, s_5 , s_6, s_7\}$. Then the non-conjugate parabolic subgroups $W_J$ and $wW_Iw^{-1}$ intersect trivially, where 
$$w = s_6 s_2 s_4 s_5 s_3 s_4 s_1 s_3 s_2 s_4 s_5 s_6 s_7 s_4 s_5 s_6 s_4 s_5 s_3 s_4 s_1 s_3 s_2 s_4 s_5 s_4 s_3 s_2 s_4 s_1. $$
This was checked using GAP.
\end{enumerate}
\end{remark}

%\begin{thomas}
%In the type $E_7$ example I did not understand why the fact that the two parabolic subgroups intersect trivially implies that they are not conjugate. 
%\end{thomas}

%\begin{patrick}
%sorry, it doesn't follow from the fact that they intersect trivially. I just added this property, because the max parabolic subgroups in $D_4$ and $D_5$, which intersect trivially, are conjugated.
%\end{patrick}

\medskip

For the rest of this section we work with the combinatorial realization of $W$ as a subgroup (which we denote by $W_{D_n}$) of the hyperoctahedral group $W_{B_n}$ (see Section \ref{sec:quasi}). To this end, set
\begin{align*}
s_0 & = (1,-2)(-1,2)\\
s_i & = (i,i+1)(-i,-(i+1))~\text{for $i \in [n-1]$}.
\end{align*}
Then $\{s_0, s_1,\dots, s_{n-1}\}$ is a simple system for a Coxeter group $W_{D_n}$ of type $D_n$. The set of reflections is given by $T=\{(i,j)(-i,-j)~|~i,j\in[-n,n], i\neq\pm j\}$. Notice that $W$ is a subgroup of a group $W_{B_n}$ of type $B_n$; 
%indeed, one defines 
%$$W_{B_n}:=\{w\in S_{-n,n}~|~w(-i)=-w(i),~\forall i\in[-n,n]\}$$
indeed, the above generators are clearly contained in $W_{B_n}$. Given $A\subset[-n,n]$, write $\stab(A)$ for the subgroup of $W_{D_n}$ of elements preserving the set $A$. Notice that since $W_{D_n}\subseteq W_{B_n}$, we have $\stab(A)=\stab(-A)$. The maximal standard parabolic subgroups of $W_{D_n}$ are then described as follows (see \cite[Prop. 8.2.4]{BB05}). Let $i \in \{0,1,\ldots,n-1\}$ and $I=S \setminus \{s_i\}$. Then $W_I=\stab(A_I)$, where\begin{align*}
   A_I :=
   \begin{cases}
     [i+1,n] & \text{if } i \neq 1 \\
     \{ -1,2,3,\ldots,n\} & \text{if } i=1.
   \end{cases}
\end{align*}
Since $W_I$ stabilizes both $A_I$ and $-A_I$, it stabilizes also the complement $A_I^0$ of $A_I\cup (-A_I)$ in $[-n,n]$. Notice that $A_I^0=-A_I^0$.

From this description we can easily achieve a description of maximal parabolic subgroups:
\begin{Lemma} \label{Le:AppB}
If $W_J=\stab(A_J)$ is a maximal standard parabolic subgroup and $w \in W$, then $wW_Jw^{-1}= \stab(w(A_J))$.
\end{Lemma}
\begin{proof}
Clear.
\end{proof}

%\begin{proof}
%Let $J' = \{ w(i_1), \ldots , w(i_m)\}$ and $k \in J'$, i.e. $k=w(i_{k'})$ for some $k' \in [m]$. 
%\begin{thomas}
%There are two $k$'s with different meanings. 
%\end{thomas}
%Then for each $v \in W_J$ it is 
%$wvw^{-1}(k)=w\underbrace{v(i_{k'})}_{\in \{ i_1, \ldots, i_m\}} \in J'$, hence 
%$wW_jw^{-1} \subseteq \stab(J')$. Now let $x \in \stab(J')$, i.e. 
%$x(w(i_{k'}))= w(i_{k''})$ for some $k',k'' \in [m]$. But then 
%$w^{-1}xw(i_{k'})= i_{k''}$, hence $x \in wW_Jw^{-1}$.
%\end{proof}

\begin{proof}[Proof of Proposition~\ref{Prop:AppB}]
It is enough to show that $W_I \cap wW_Jw^{-1} \neq \{ \idop \}$ for $I,J \subseteq S$ with $\vert I \vert = \vert J \vert = n-1$ and $w \in W$. Consider the intersections $A_I \cap w(A_J)$ and $A_I \cap -w(A_J)$. If one of these intersections contains at least two elements, say $k$ and $l$, then $(k,l)(-k,-l) \in W_I \cap wW_Jw^{-1}$ since $A_I\cap (-A_I)=\varnothing=w(A_J)\cap(-w(A_J))$. 
Therefore we can assume that $\vert A_I \cap w(A_J) \vert \leq 1$ and $\vert A_I \cap -w(A_J) \vert \leq 1$. 
Now if $|A_I|\geq 4$, then $|A_I\cap w(A_J)^0|\geq 2$, and since $A_I\cap(-A_I)=\varnothing$ it follows that there exist $k,\ell\in A_I\cap w(A_J)^0$ with $k\neq\pm\ell$, and we then have that $(k,\ell)(-k,-\ell)\in W_I \cap wW_Jw^{-1}$. Hence we can furthermore assume that $|A_I|<4$. It follows that $|A_I^0|\geq 2n-6$. 

But arguing similarly we can also assume that $|A_I^0\cap w(A_J)^0|<4$, hence $|A_I^0\cap w(A_J)^0|\leq 2$ since it has to be even and $|A_I^0\cap w(A_J)|\leq 1$. It follows that $|A_I^0|\leq 4$. Together with the inequality above we get $2n-6\leq |A_I^0|\leq 4$, hence $n\leq 5$.

\end{proof}

\section{The proof of Theorem~\ref{th:maintheorem1}}\label{sec:proof}

\noindent The aim of this section is to prove the sufficient condition for the transitive Hurwitz action as stated in Theorem \ref{th:maintheorem1}. The proof is a case-by-case analysis. 

Let $(W,T)$ be a finite dual Coxeter system and let $w\in W$. For $(t_1, \ldots , t_m), (r_1, \ldots, r_m) \in \Red_T(w)$ we write $(t_1, \ldots , t_m) \sim (r_1, \ldots, r_m)$ if both factorizations lie in the same Hurwitz orbit. Furthermore note that if $w$ is a parabolic quasi-Coxeter element in $(W,T)$, then each conjugate of $w$ is also a parabolic quasi-Coxeter element in $(W,T)$.  Since the Hurwitz operation commutes with conjugation, we can restrict ourselves to check transitivity for one representative of each conjugacy class of parabolic quasi-Coxeter elements of $W$. The proof of the following is easy:

\begin{Lemma}\label{product}
Let $(W_i, T_i)$, $i=1,2$ be dual Coxeter systems and let $w_i\in W_i$, $i = 1,2$. Then $(W_1 \times W_2, T:=(T_1 \times\{1\}) \cup (\{1\}\times T_2))$ is a dual Coxeter system. Furthermore, if the Hurwitz action is transitive
on $\Red_{T_i}(w_i)$, $i = 1,2$, then the Hurwitz action is transitive on 
$\Red_{T}( (w_1, w_2) )$.
\end{Lemma}

\subsection{Types $A_n$, $B_n$ and $I_2(m)$}\label{sub:ABI}

In all these cases every parabolic quasi-Coxeter element is already a parabolic Coxeter element
by Lemmas~\ref{QuasiCoxIsCoxA}, \ref{lem:coxeterb} and Corollary~\ref{QuasiCoxIsCoxI}.
Therefore the assertion follows with Theorem~1.3 of \cite{BDSW14}. 

\begin{remark}
Notice that in types $A_n$ and $B_n$, parabolic quasi-Coxeter elements, classical parabolic Coxeter elements and parabolic Coxeter elements coincide. In type $I_2(m)$, parabolic quasi-Coxeter elements and parabolic Coxeter elements coincide, but parabolic Coxeter elements and classical parabolic Coxeter elements do not (see Remark \ref{rmk:coxeter} $(c)$).
\end{remark}

\subsection{The simply laced types.}

We now treat the parabolic quasi-Coxeter elements in an irreducible, finite, simply laced dual Coxeter system $(W, T)$ of rank $n$. As we already dealt with the type $A_n$, it remains to consider the types $D_n$, $n\geq 4$ and $E_6, E_7, E_8$.

\medskip

We only need to show the assertion for quasi-Coxeter elements. Indeed, let $w$ be a parabolic quasi-Coxeter element in $(W,T)$ and let $(t_1, \ldots , t_m) \in \Red_T(w)$. Then 
$W^\prime = \langle t_1, \ldots, t_m \rangle$ is by Theorem \ref{ThmA} a parabolic subgroup of $(W,T)$, in fact $W^\prime = P_w$.

Therefore, it follows from Lemma \ref{lem:inclusionfixed} that all the reflections in any reduced factorization
of $w$ are in $W^\prime$. The latter group is a direct product of irreducible Coxeter groups of simply laced type.
If we know that the Hurwitz action  is transitive on $\Red_T(\tilde{w})$ for all the quasi-Coxeter elements $\tilde{w}$ in the irreducible Coxeter groups of simply laced type, then the Hurwitz action on $\Red_T(w)$ is transitive as well by Lemma~\ref{product}.

\medskip

The strategy to prove the theorem is as follows: we first show by induction on the rank $n$ (with $n\geq 4$) that the Hurwitz action is transitive on the set of reduced decompositions of quasi-Coxeter elements of type $D_n$; for this we will need to use the result for parabolic subgroups, but since they are (products) of groups of type $A$ with groups of type $D$ of smaller rank, the result holds for groups of type $A$ by Subsection \ref{sub:ABI} and they hold for groups of type $D_k$, $k<n$ by induction. 

Using the fact that it holds for type $D_n$, $n\geq 4$, we then prove the result for the groups $E_6$, $E_7$ and $E_8$. Similarly as for type $D_n$, parabolic subgroups of type $E$ are of type $A$, $D$ or $E$. It was previously shown to hold for types $A$ and $D$ and holds for type $E$ by induction. 

We then prove by computer that the result holds for  the remaining exceptional groups. 

Let $w$ be a quasi-Coxeter element and let $(t_1,  \ldots , t_n) \in \Red_T(w)$. 

\medskip

\subsubsection{Type $D_n$}
For types $D_4$ and $D_5$ the assertion is checked directly using \cite{GAP2015}. Therefore assume  $n \geq 6$. Let $(r_1, \ldots ,r_n) \in \Red_T(w)$ be a second reduced factorization of $w$. By Theorem \ref{ThmA} and Theorem \ref{thm:quasi-CoxeterParabolic} the groups $\langle t_1, \ldots , t_{n-1} \rangle$ and $\langle r_1, \ldots , r_{n-1} \rangle$ are maximal parabolic subgroups and 
since $\lt(w t_n)=n-1=\lt(wr_n)$ it follows that $C_W(V^{wt_n})=P_{wt_n}=\langle t_1, \ldots , t_{n-1} \rangle$, $C_W(V^{wr_n})=P_{wr_n}=\langle r_1, \ldots , r_{n-1} \rangle$. By Proposition \ref{Prop:AppB} there exists a reflection $t$ in their intersection. It follows by Lemma~\ref{lem:inclusionfixed} that $t <_T wt_n, wr_n$. Hence there exists $t_2',\dots, t_{n-1}',r_2',\dots, r_{n-1}'\in T$ such that $(t,t_2',\ldots , t_{n-1}') \in \Red_T(wt_{n})$ and $(t,r_2', \ldots, r_{n-1}') \in \Red_T(wr_{n})$. In particular we get 
$$(t_2',\ldots , t_{n-1}',t_n), (r_2', \ldots, r_{n-1}',r_n) \in \Red_T(tw).$$
By Theorem~\ref{thm:quasi-CoxeterParabolic} the element $tw$ is 
quasi-Coxeter in the parabolic subgroup $$P_{tw}= \langle t_2',\ldots , t_{n-1}',t_n \rangle.$$ It follows from Lemma \ref{lem:inclusionfixed} that the reflections $r_2', \ldots, r_{n-1}',r_n$ are in $P_{tw}$ since $r_i'<_{T} tw$ for each $i$. As $P_{tw}$ is a direct product of irreducible Coxeter groups of type $A$ and $D$ of smaller rank, we have by induction together with Lemma~\ref{product} that
\begin{align*}
(t_2^\prime, \ldots, t_{n-1}^\prime,t_n) \sim (r_2^\prime, \ldots , r_{n-1}^\prime,r_n),
\end{align*}

as well as 

\[ (t,t_2^\prime, \ldots , t_{n-1}^\prime) \sim (t_1, \ldots, t_{n-1})~\mbox{and}~(t,r_2^\prime, \ldots, r_{n-1}^\prime)
 \sim (r_1, \ldots, r_{n-1}).\]

This implies 

\begin{align*}
(t_1, \ldots, t_n) \sim (t,t_2^\prime, \ldots , t_{n-1}^\prime,t_n) \sim (t,r_2^\prime, \ldots, r_{n-1}^\prime, r_n) \sim (r_1, \ldots, r_n)\in\Red_T(w),
\end{align*}

which concludes the proof.

\subsubsection{Types $E_6, E_7$ and $E_8$}

We checked using \cite{GAP2015} that there is a reduced factorization $(t_1, \ldots , t_n)$ of the quasi-Coxeter element $w$ such that for every reflection $t$ in $T$
 there exists $(t_1^\prime, \ldots , t_{n-1}^\prime, t) \in \Red_T(w)$ with $(t_1, \ldots ,t_n) \sim (t_1^\prime, \ldots , t_{n-1}^\prime, t)$.

\medskip

Let $(r_1, \ldots , r_n) \in \Red_T(w)$. By our computations in GAP there exists
$(t_1^\prime, \ldots , t_{n-1}^\prime, r_n) \in \Red_T(w)$ with 
$(t_1, \ldots , t_n) \sim (t_1^\prime , \ldots , t_{n-1}^\prime, r_n)$. 
Then
\begin{align*}
 wr_n = t_1^\prime \cdots t_{n-1}^\prime = r_1 \cdots r_{n-1}
\end{align*}
are reduced factorizations. Furthermore $wr_n$ is a quasi-Coxeter element in $(W^\prime, T^\prime)$ where 
$W^\prime := \langle t_1^\prime , \ldots , t_{n-1}^\prime \rangle$ and $T^\prime := T \cap W'$. By Theorem~\ref{thm:quasi-CoxeterParabolic} we have that $W^\prime$ is a equal to the parabolic closure $P_{t_1'\cdots t_{n-1}'}$ of $t_1'\cdots t_{n-1}'$
and therefore $r_1, \ldots , r_{n-1} \in W^\prime$ by Lemma \ref{lem:inclusionfixed}. Thus
$(t_1^\prime, \ldots , t_{n-1}^\prime)$ and $(r_1 , \ldots , r_{n-1})$ are reduced factorizations of a quasi-Coxeter element in a dual, simply laced Coxeter system of rank $n-1$.
By induction and by Lemma~\ref{product} we get
$(t_1^\prime, \ldots , t_{n-1}^\prime) \sim (r_1 , \ldots , r_{n-1})$, thus
\begin{align*}
 (t_1, \ldots , t_{n})  \sim (t_1^\prime , \ldots , t_{n-1}^\prime, r_n)\sim (r_1 , \ldots , r_{n}).
\end{align*}

\subsection{The types $F_4$, $H_3$ and $H_4$}
For these cases Theorem \ref{th:maintheorem1} was checked directly using \cite{GAP2015}.

%\begin{proof}
%Let $W=W_{\Phi}$ and assume $W=\langle s_{\beta_1}, \ldots , s_{\beta_k} \rangle$ for some $k<n$ and $\{ \beta_1 , \ldots , \beta_k \} \subseteq \Phi$ linearly independent. Let $w= s_{\beta_1} \cdots  s_{\beta_k}$. Thus $\spanr(\Phi) = 
%\text{Mov}(w) \oplus \text{Fix}(w)$. Since 
%$\dim(\text{Mov}(w))= l_T(w) =k < n $, there exists a simple root $\alpha \in \text{Fix}(w)$. By assumption we can write 
%$s_{\alpha}= s_{\beta_{i_1}} \cdots  s_{\beta_{i_l}}$ with $1 \leq i_j \leq k$. Since 
%$\text{Mov}(w) = \text{Fix}(w)^{\bot}$, it is 
%\begin{align*}
%\alpha = s_{\beta_{i_1}} \cdots  s_{\beta_{i_l}}(\alpha) = s_{\alpha}(\alpha)= - \alpha,
%\end{align*}
%a {\color{red}contradiction}.
%\end{proof}

%{\color{red}

  \end{document}